\documentclass{amsart}

\usepackage{graphicx}
\usepackage[numbers]{natbib}
\usepackage{color,soul}
\usepackage[all]{xy}
\usepackage{subfigure}
\usepackage{geometry,mathtools}
\usepackage{amsthm}
\usepackage{float}
\usepackage{tikz-cd}
\usepackage{extarrows}
\usepackage{extarrows}
\usepackage{multirow}
\usepackage{bm}
\usepackage{amssymb}
\usepackage{soul}
\usepackage{graphicx}
\usepackage{comment}

\newtheorem{theorem}{Theorem}[section]

\newtheorem{proposition}[theorem]{Proposition}

\newtheorem{lemma}[theorem]{Lemma}
\newtheorem{prop}[theorem]{Proposition}

\theoremstyle{definition}
\newtheorem{definition}[theorem]{Definition}
\newtheorem{example}[theorem]{Example}

\newtheorem{corollary}[theorem]{Corollary}

\theoremstyle{remark}
\newtheorem{remark}[theorem]{Remark}

\numberwithin{equation}{section}

\newcommand{\tabincell}[2]{\begin{tabular}{@{}#1@{}}#2\end{tabular}}  
\newcommand{\Mod}[1]{\ (\mathrm{mod}\ #1)}

\begin{document}

\title{Surgeries between lens spaces of type $L(n,1)$ and the Heegaard Floer $d$-invariant}
%\title{Distance One Surgery and the Realization Problems for Lens Spaces of Type $L(n,1)$}

\author{Zhongtao Wu}
%    Address of record for the research reported here
\address{Department of Mathematics, The Chinese University of Hong Kong}
\email{ztwu@math.cuhk.edu.hk}
\author{Jingling Yang}
\address{School of mathematics, Xiamen University}
\email{yangjingling@xmu.edu.cn}

\thanks{} %The first author was supported in part by NSF Grant \#000000.

%    Information for second author
%\author{Author Two}
%\address{Mathematical Research Section, School of Mathematical Sciences,
%Australian National University, Canberra ACT 2601, Australia}
%\email{two@maths.univ.edu.au}
%\thanks{Support information for the second author.}

%    General info
%\subjclass[2000]{Primary 54C40, 14E20; Secondary 46E25, 20C20}

%\date{January 1, 2001 and, in revised form, June 22, 2001.}

%\dedicatory{This paper is dedicated to our advisors.}


\begin{abstract}

We establish a $d$-invariant surgery formula for $L$-space knots that provides an effective tool for studying surgeries between lens spaces. Using this formula, we classify distance one surgeries between lens spaces of the form $L(n,1)$. %The $d$-invariant formula, which may find broader applications in low-dimensional topology, allows us to precisely determine when such surgeries are possible. 
This classification has direct applications to band surgeries between torus links $T(2,n)$, with connections to DNA topology. In particular, we show that chirally cosmetic banding of torus links can possibly occur only when $n=1,5,9$ or $10$.

\end{abstract}

\maketitle

\section{Introduction}

%The lens space realization problem, addressed by Greene \cite{Greene}, classifies all lens spaces that can be obtained by Dehn surgery along a knot in $S^3$. This problem is part of the broader Berge conjecture\cite{Berge,Kirby}, which remains open and attracted significant attention in recent years \cite{Bleiler,Boil,Goda,OS1,OSkl,Ras3,Ras2}. A natural generalization of this inquiry involves determining which lens spaces can be obtained via Dehn surgery along a knot in another lens space.

The study of relationships between $3$-manifolds under Dehn surgery is a central theme in low-dimensional topology. A particularly interesting case concerns surgeries between lens spaces.  While the lens space surgery problem on $S^3$ has been extensively studied - notably through Greene's work \cite{Greene} and its connections to the Berge conjecture \cite{Berge,Kirby}, the more general question of surgeries between arbitrary lens spaces remains largely unexplored. This paper addresses this gap by providing a systematic classification for a specific but important family: distance one surgeries between lens spaces of the form $L(n,1)$.

According to the cyclic surgery theorem, surgeries between lens spaces with distance greater than one can be resolved. This directs our attention to distance one surgeries, where the meridian of the surgered knot and the surgery slope have minimal geometric intersection number one.
In this setting, Lidman, Moore and Vazquez \cite{LMV} made initial progress by identifying which lens spaces $L(s,1)$ can be obtained by a distance one surgery on $L(3,1)$. Building upon their work, the authors obtained partial results concerning lens spaces $L(s,1)$ arising from distance one surgery on $L(p,1)$ for prime $p$ \cite{WY}.  The present paper significantly extends these results by characterizing all possible pairs $(n,s)$ for which $L(s,1)$ might be obtained by a distance one surgery on $L(n,1)$.

Our focus on $L(n,1)$ is motivated in part by its connection to the band surgery problem.  A {\it band surgery} on a knot or link $L$ involves embedding a unit square $I \times I$ into $S^3$ by $b:I \times I \rightarrow S^3$ such that $L \cap b(I \times I)=b(\partial I \times I)$, then replacing $L$ by $L'= (L-b(\partial I \times I)) \cup b(I \times \partial I)$. The Montesinos trick relates band surgeries on $T(2,n)$ torus links to distance one surgeries on their double branched covers $L(n,1)$, making these lens spaces particularly relevant to the study of band surgeries between torus links.
%Since $L(n,1)$ is the double branched cover of the torus link $T(2,n)$, the Montesinos trick implies that band surgeries on these links lift to distance one surgeries in their double branched covers.  

This connection to band surgery has significant implications beyond pure topology.  In molecular biology, circular DNA molecules are naturally modeled as knots or links, with $T(2,n)$ being a family that frequently appears in biological experiments. DNA recombination, where strands of DNA are exchanged through enzymatic processes, can be mathematically modeled as band surgery. Thus, our study of distance one surgeries between lens spaces of the form $L(n,1)$ not only advances our understanding of $3$-manifold topology but also provides insights into biological processes.

%Why focus on $L(n,1)$ (or $T(2,n)$)? Besides being the simplest lens space (respectively, the simplest 2-bridge link), these constructions are motivated by DNA topology. In biology, circular DNA can be modeled as a knot or link, and torus knots or links $T(2,n)$ are a family of DNA knots or links occurring frequently in biological experiments. Additionally, there exist enzymatic complexes that mediate DNA recombination, during which strands of DNA are exchanged, and the topology of the DNA molecule may be altered in the process. The mechanism of DNA recombination can be modeled by band surgery, providing a biological motivation for studying distance one surgeries between lens spaces of type $L(n,1)$ and band surgeries between torus links of type $T(2,n)$.

\begin{figure}[t]
\centering
\subfigure[Band surgery from $T(2,n)$ to $T(2,n)$.]{\label{bandsurgery1}
\includegraphics[width=0.205\textwidth]{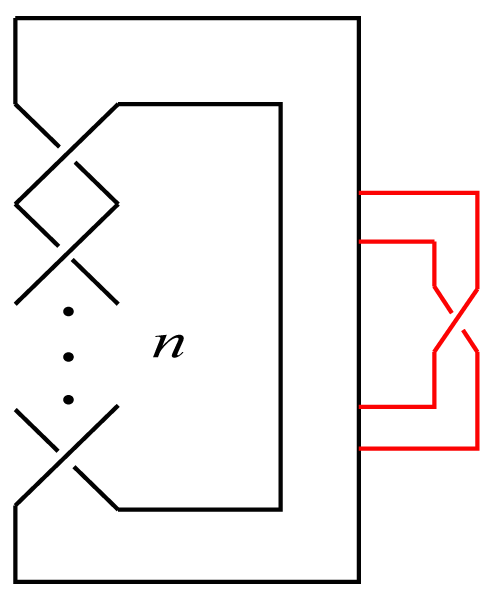}}\quad
% \subfigure[Band surgery from $T(2,n)$ to $T(2,n+1)$.]{\label{bandsurgery3}
% \includegraphics[width=0.2\textwidth]{pictures/bandsurgery3.png}}\quad
\subfigure[Band surgery from $T(2,n)$ to $T(2,n-4)$.]{\label{bandsurgery5}
\includegraphics[width=0.25\textwidth]{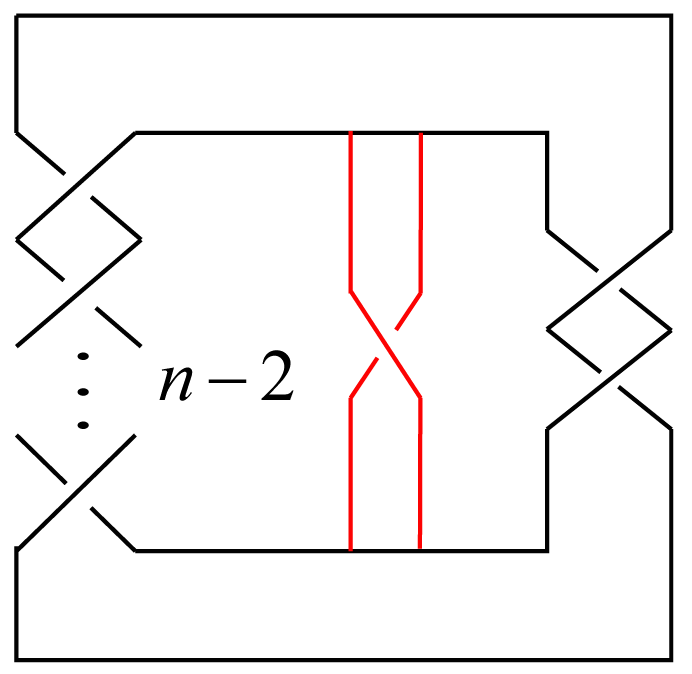}}\quad
\subfigure[Band surgery from $T(2,5)$ to $T(2,-5)$ constructed by Zekovi\'{c} (see \cite{Zek} and \cite{MV}).]{\label{bandsurgeryT25}
\includegraphics[width=0.32\textwidth]{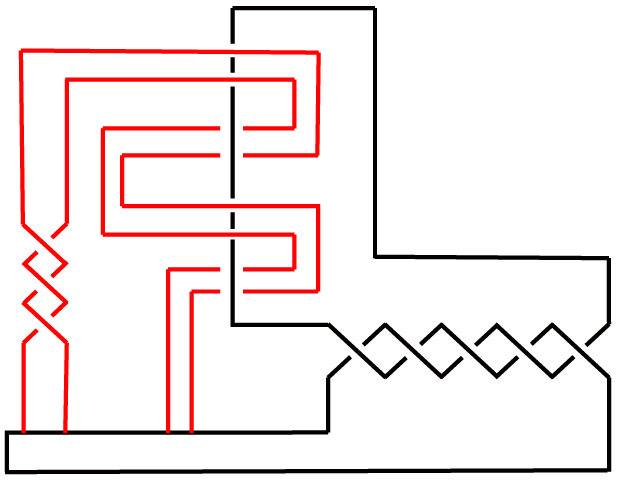}}\\
\subfigure[Band surgery from $T(2,n)$ to $T(2,n-1)$.]{\label{bandsurgery2}
\includegraphics[width=0.2\textwidth]{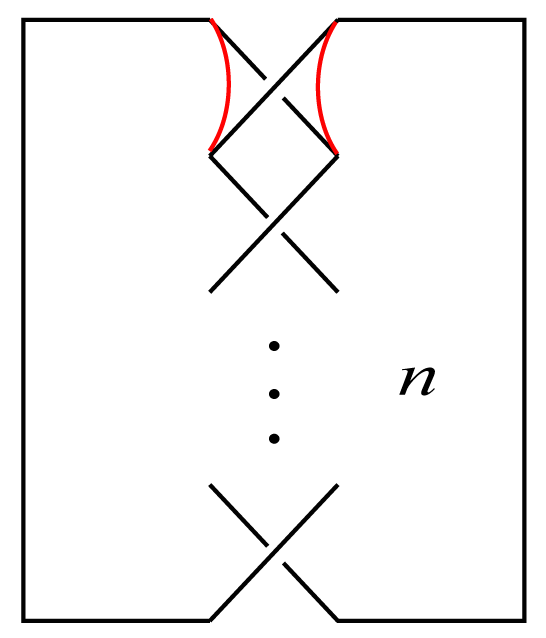}}\quad
\subfigure[Band surgery from $T(2,n)$ to the unknot.]{\label{bandsurgery4}
\includegraphics[width=0.196\textwidth]{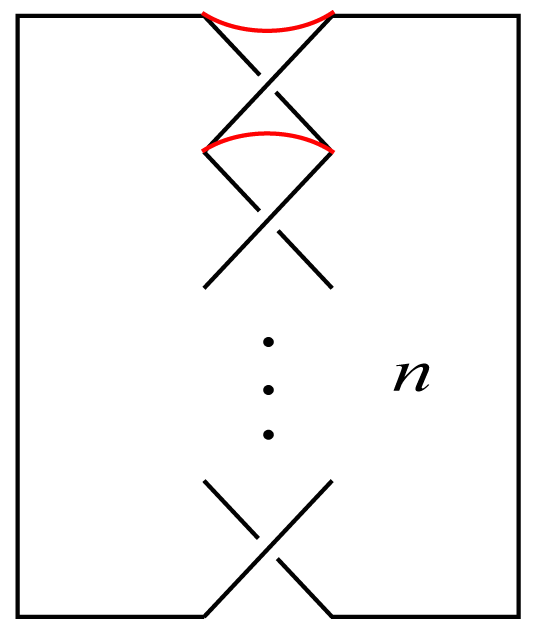}}
\subfigure[Band surgery from $T(2,6)$ to $T(2,-3)$ (see \cite{LMV}).]{\label{bandsurgeryT26}
\includegraphics[width=0.375\textwidth]{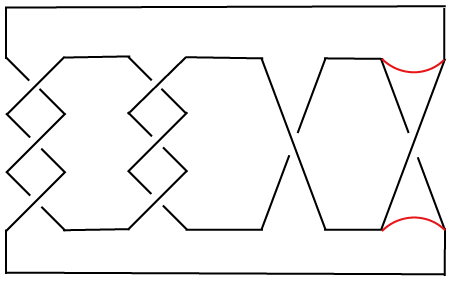}}

\caption{Examples of band surgeries, which lift to distance one surgeries in double branched covers.}\label{bandsurgery}
\end{figure}

Our main result is as follows. For simplicity, we restrict our attention to the case $0<|s| \leq n$, which is sufficient due to the symmetry inherent in distance one surgeries. While we identify all potential pairs $(n,s)$, we note that the existence of surgeries remains unresolved in four specific cases, corresponding to cases (6)-(9) in our main theorem.

\begin{theorem}
\label{L(n,1)maintheorem}
The lens space $L(s,1)$ can be obtained from a distance one surgery along a knot in $L(n,1)$ with $n\geq |s|>0$ only if $n$ and $s$ satisfy one of the following cases:
\begin{itemize}
\item[(1)] $n$ is any positive integer and $s=\pm 1, n, n-1$ or $n-4$;
\item[(2)] $n=3$ and $s=-2$;
\item[(3)] $n=5$ and $s=-5$;
\item[(4)] $n=6$ and $s=-2$;
\item[(5)] $n=6$ and $s=-3$;
\item[(6)] $n=9$ and $s=-5$;
\item[(7)] $n=9$ and $s=-9$;
\item[(8)] $n=10$ and $s=-10$;
\item[(9)] $n=14$ and $s=-10$.
\end{itemize}
\end{theorem}

\begin{remark}
\label{remarkintro}
Except for the cases (6)-(9), distance one surgeries for all other pairs of $n$ and $s$ in the above list can be realized as the double branched covers of the band surgeries illustrated in Figure \ref{bandsurgery}. In particular, since $L(-2,1) \cong L(2,1)$, the distance one surgeries in cases (2) and (4) are equivalent to distance one surgeries from $L(3,1)$ to $L(2,1)$ and from $L(6,1)$ to $L(2,1)$, respectively, both of which fall under case (1). %The existence of distance one surgeries in cases (6)-(9) remains unknown.

\end{remark}

The above theorem implies the following corollary regarding band surgery once we lift to the double branched covers.

\begin{corollary}
\label{T(2,n)corollary}
The torus link $T(2,s)$ cannot be obtained by a band surgery from $T(2,n)$ with $n \geq |s|>0$ unless 
$n$ and $s$ satisfy one of the following cases:
\begin{itemize}
\item[(1)] $n$ is any positive integer and $s=\pm 1, n, n-1$ or $n-4$;
\item[(2)] $n=3$ and $s=-2$;
\item[(3)] $n=5$ and $s=-5$;
\item[(4)] $n=6$ and $s=-2$;
\item[(5)] $n=6$ and $s=-3$;
\item[(6)] $n=9$ and $s=-5$;
\item[(7)] $n=9$ and $s=-9$;
\item[(8)] $n=10$ and $s=-10$;
\item[(9)] $n=14$ and $s=-10$.
\end{itemize}
\end{corollary}

%{\it Chirally cosmetic banding.} 
If a band surgery relates a knot to its mirror image, it is called {\it chirally cosmetic banding}. Zekovi\'{c} \cite{Zek} found a chirally cosmetic banding for the torus knot $T(2,5)$ as shown in Figure \ref{bandsurgeryT25}. Moore and Vazquez \cite[Corollary 4.5]{MV} showed that for all other torus knots $T(2,m)$ with $m>1$ \textit{odd} and \textit{square free}, there do not exist any chirally cosmetic bandings. Livingston \cite[Theorem 14]{Livingston} further extends this result using Casson-Gordon theory, removing the square-free constraint on $m$, except for $m=9$. Theorem \ref{L(n,1)maintheorem} leads to the following corollaries, which enhance the results of Moore, Vazquez and Livingston, by including all torus links $T(2,n)$ with $n>0$.

\begin{corollary}
\label{chirallycosmeticsurgery}
There exists a distance one surgery along a knot $K$ in the lens space $L(n,1)$ with $n>0$ that yields $-L(n,1)$ only if $n=1$, $5$, $9$ or $10$. 
\end{corollary}

\begin{corollary}
\label{chirallycosmeticband}
If the torus link $T(2,n)$ with $n>0$ admits a chirally cosmetic banding, then $n=1,5,9$ or $10$.
\end{corollary}

Our primary obstruction technique involves a detailed analysis of Heegaard Floer mapping cones of $L$-space knots in $L$-spaces.  While a general $d$-invariant surgery formula for $d(Y_\gamma(K),\mathfrak{s})$ %to the full generality of \cite{NiWu} 
remains elusive, we establish several key results for L-space surgeries, summarized in the following theorem, which is a special case of the more general results proved in Theorems \ref{Thm:L-spaceknotpartialorder}, \ref{Thm:L-spaceSurgeryFormula} and Proposition \ref{CassonWalkerObstruction}:

\begin{theorem}\label{Thm:Main}
Let $K$ be a knot in an $L$-space $Y$ equipped with positive framing $\gamma$, and let $K'$ be a Floer simple knot in $Y$ such that $[K]=[K'] \in H_1(Y)$.
If $\gamma$-surgery along $K$ produces an $L$-space $Y_{\gamma}(K)$, then:
\begin{enumerate}
\item The surgered manifold $Y_\gamma(K')$ is also an $L$-space.

\item For any $\mathfrak{s}\in  {\rm Spin^c}(Y_\gamma(K))\cong {\rm Spin^c}(Y_\gamma(K'))$, there exists a non-negative integer $N$ such that
$$d(Y_\gamma(K), \mathfrak{s})=d(Y_\gamma(K'), \mathfrak{s})-2N.$$

\item The Casson-Walker invariants satisfy 
$$\lambda(Y_\gamma(K))\geq \lambda(Y_\gamma(K')).$$ 

\end{enumerate}
    
\end{theorem}

\subsection*{Organization} The paper is organized as follows: Section \ref{Preliminaries} reviews essential homological background relevant to Dehn surgeries of knots in lens spaces. Section \ref{Preliminaries in Heegaard Floer homology} discusses the fundamentals of Heegaard Floer homology, including the $d$-invariant, the mapping cone formula, $L$-space knots and Rasmussen's notation. Section \ref{d-invariant surgery formula for homologically essential knots} establishes a general $d$-invariant surgery formula for $L$-space surgeries, as stated in Theorem \ref{Thm:L-spaceSurgeryFormula}. We then apply this to obtain a $d$-invariant surgery formula for $L(n,1)$.  Section 5 introduces the closely related Casson-Walker invariants, computes their values for certain Seifert fibered spaces, and establishes a constraint on surgeries using Casson-Walker invariants. Section \ref{Proof of main theorem} outlines the proof of Theorem \ref{L(n,1)maintheorem}, highlighting the main arguments and techniques used in various cases. Finally, Sections \ref{Distance one surgeries between $L(n,1)$ and $L(s,1)$ with $n$ and $s$ different parities}-\ref{$K$ is homologically essential} employ our $d$-invariant surgery formulas to investigate different cases of the distance one surgery problem.

\subsection*{Acknowledgments}
We would like to thank Guanheng Chen for valuable discussions. The first author is partially supported by a grant from the Research Grants Council of Hong Kong Special Administrative Region, China (Project No. 14301819) and a direct grant from CUHK (Project No. 4053661). The second author is supported by Fundamental Research Funds for the Central Universities (Project No. 20720230026), the National Natural Science Foundation of China (Project No. 12301087) and Fujian Provincial Natural Science Foundation of China (Project No. 2024J08012).

\section{Homological preliminaries}
%and the four-dimensional perspective}
\label{Preliminaries}
In our convention, $L(p,q)$ denotes the lens space obtained from $p/q$-surgery along the unknot in $S^3$. Let $K$ be a knot in $Y=L(n,1)$ with $n>0$. A knot $K$ in $Y$ is called \emph{null-homologous} if its homology class is trivial in $H_1(Y)$, otherwise $K$ is called \emph{homologically essential}. 

For a null-homologous knot $K$, there exists a canonical choice of meridian $\mu$ and longitude $\lambda$. Let $Y_{m}(K)$ denote the manifold obtained by $(m\mu+\lambda)$-surgery along $K$. Then $H_1(Y_m(K))=\mathbb{Z}_n \oplus \mathbb{Z}_{|m|}$. If this homology group is cyclic, then the greatest common divisor of $n$ and $m$ must be $1$, and
\begin{equation} \label{eqnhkh}
H_1(Y_m(K))=\mathbb{Z}_n \oplus \mathbb{Z}_{|m|}=\mathbb{Z}_{|mn|}.
\end{equation}

For a homologically essential knot $K$, its homology class $[K] \in H_1(Y) = \mathbb{Z}_n$ can be identified with an integer in $\{1, \cdots , \lfloor n/2 \rfloor \}$ after fixing a generator of $\mathbb{Z}_n$. To make this precise, we describe $Y$ via a Kirby diagram consisting of an unknot $U$ with framing $n$ in $S^3$. The manifold $Y$ is obtained by gluing a solid torus $V_0$ to the unknot complement $V_1$. Let $c$ denote the core of $V_1$, oriented such that the linking number $\text{lk}(c,U) = 1$. %After possible handle-slides of $K$ over $U$, which is equivalent to isotoping $K$ over the meridian of $V_0$, 
After possible isotopy over the meridian of $V_0$ and orientation reversal, we may assume $[K] = k[c]$ where $k = \text{lk}(K,U)$ is an integer with $1 \leq  k \leq  \lfloor n/2 \rfloor$. We call this integer $k$ the \emph{mod n-winding number}, or simply the \emph{winding number} of $K$.

To describe the peripheral system, we consider $K \cup U$ as a link in $S^3$ and choose meridians $\mu, \mu_0$ and longitudes $\lambda, \lambda_0$ for $K$ and $U$ respectively. The first homology of the link complement can then be described as:
$$H_1(S^3-U-K)=\mathbb{Z} \langle \mu_0 \rangle \oplus \mathbb{Z} \langle \mu \rangle \quad {\rm and} \quad \left \{\begin{matrix}
[\lambda_0]=k \cdot [\mu]\\
[\lambda]=k \cdot [\mu_0]
\end{matrix}\right. .$$
Therefore, the first homology of the knot complement $Y-K$ is
$$H_1(Y-K)=H_1(S^3-U-K)/ \langle n\mu_0 + \lambda_0 \rangle= H_1(S^3-U-K)/ \langle n\mu_0 + k \mu \rangle .$$
Let $\theta= n'\mu_0+k'\mu$ and $\vartheta=\frac{n}{d} \mu_0+\frac{k}{d}\mu$, where $d={\rm gcd}(n,k)$ and $\frac{n}{d}k'-\frac{k}{d}n'=1$. Then
\[H_1(Y-K)=\mathbb{Z} \langle \theta \rangle \oplus \mathbb{Z}_d \langle \vartheta \rangle.\] 
One may check that
\begin{align}
[\mu]&=\frac{n}{d}[\theta]-n'[\vartheta] \in H_1(Y-K), \label{mu2021}\\
[\lambda]&=\frac{-k^2}{d}[\theta]+kk'[\vartheta] \in H_1(Y-K). \label{lambda2021}
\end{align}
For distance one surgery, we consider $(m \mu + \lambda)$-surgery. We have
\begin{equation}
\label{mmu+lambda2021}
[m\mu+ \lambda]=\frac{mn-k^2}{d}[\theta]+(kk'-n'm)[\vartheta]\in H_1(Y-K).
\end{equation}
Denote by $Y_{m\mu+ \lambda}(K)$ the manifold obtained by $(m\mu + \lambda)$-surgery along $K$. Then 

\begin{equation} 
\label{eqmn-k2}
|H_1(Y_{m\mu+ \lambda}(K))|=\left|det 
 \begin{bmatrix}
 \dfrac{mn-k^2}{d} & kk'-n'm \\
 0 & d \\
\end{bmatrix}\right|=|mn-k^2|.
\end{equation}
If this homology group is cyclic, the Smith normal form of this matrix must be 
\begin{equation*}
 \begin{bmatrix}
1 & 0 \\
0 & mn-k^2 \\
\end{bmatrix}.
\end{equation*}
Thus, 
\begin{equation}\label{gcd}
    \gcd\left(\frac{mn-k^2}{d},\, kk'-n'm,\, d\right)=1,
\end{equation}
since the elementary operations in the algorithm of the Smith normal form do not decrease the greatest common divisor of all entries.

For simplicity, we refer to the above surgery as the $m$-surgery along $K$, with $\mu$ and $\lambda$ chosen as described.  We now define the sign of a framing $\gamma=m\mu+\lambda$.

\begin{definition}
Let $d_1$, $d_2$ be integers such that \[d_1[\gamma]=d_2[\mu] \in H_1(Y-K).\]
We call $\gamma$ a {\it positive framing} if $d_1$ and $d_2$ have the same sign, and a {\it negative framing} if they have opposite signs. 
\end{definition}

From (\ref{mmu+lambda2021}) and (\ref{mu2021}), we have: 
\begin{lemma}\label{PositiveFraming}
For $n>0$, the framing $\gamma=m\mu+\lambda$ of a knot $K$ with winding number $k$ in $L(n,1)$ is positive (resp. negative) if and only if $mn-k^2$ is positive (resp. negative).   
\end{lemma}

\subsection{Linking form}
We conclude with a brief discussion of the \textit{linking form} 
$$lk: H_1(Y) \times H_1(Y) \rightarrow \mathbb{Q}/\mathbb{Z}$$ 
for a rational homology sphere $Y$, which will be useful as a surgery obstruction. If a rational homology sphere $Y$ has cyclic first homology, its linking form can be characterized by the value $lk(x,x)$ for a generator $x$ of $H_1(Y)$.
For example, the lens space $L(p,q)$ has linking form $q/p$, since it
is obtained by $p/q$-surgery along the unknot.
Two linking forms $q_1/p$ and $q_2/p$ for $p>0$ are isomorphic if and only if $q_1 \equiv q_2 a^2 \pmod p$ for some integer $a$ with ${\gcd}(a,p)=1$. Clearly, homeomorphic spaces 
have isomorphic linking forms. See \cite{CFH} for a detailed exposition.

\section{Preliminaries in Heegaard Floer homology}
\label{Preliminaries in Heegaard Floer homology}
In this section, we introduce some preliminaries in Heegaard Floer homology, including the $d$-invariant and the mapping cone formula, which serves as our primary tool. We use $\mathbb{F}=\mathbb{Z}/2\mathbb{Z}$ coefficients throughout, unless otherwise stated.
\subsection{$d$-invariant}
\label{d-invariant section}
For a rational homology sphere $Y$ with $\rm Spin^c$ structure $\mathfrak{t}$, the Heegaard Floer homology $HF^+(Y, \mathfrak{t})$ decomposes as an $\mathbb{F}[U]$-module
\[HF^+(Y, \mathfrak{t})=\mathcal{T}^+\oplus HF_{red}(Y, \mathfrak{t})\]
where $\mathcal{T}^+=\mathbb{F}[U,U^{-1}]/U \cdot \mathbb{F}[U]$ is called the {\it tower}, and $HF_{red}(Y, \mathfrak{t})$ is a torsion $\mathbb{F}[U]$-module. The {\it $d$-invariant} (or {\it correction term}) $d(Y, \mathfrak{t})$ is defined as the minimal $\mathbb{Q}$-grading of the tower.  These $d$-invariants satisfy: 
\begin{equation*}
d(-Y,\mathfrak{s})=-d(Y,\mathfrak{s}), \quad d(Y, J\mathfrak{s})=d(Y, \mathfrak{s})
\end{equation*}
for each $\mathfrak{s} \in {\rm Spin^c}(Y)$, where $-Y$ denotes $Y$ with opposite orientation and $J$ represents the conjugation action.

\medskip
A rational homology sphere $Y$ is called an {\it $L$-space} if 
$HF^+(Y, \mathfrak{t})=\mathcal{T}^+$
for all $\mathfrak{t} \in {\rm Spin^c}(Y)$. All lens spaces are $L$-spaces.
%We refer the readers to Ozsv\'{a}th-Szab\'{o} \cite{OS1} for further details. 
The $d$-invariant of a lens space can be computed recursively.

\begin{theorem}\cite[Proposition 4.8]{OS1}
\label{thmOS1}
For relatively prime integers $p>q>0$, there exists an identification ${\rm Spin^c}(L(p,q)) \cong \mathbb{Z}_p$ such that
\begin{equation}
\label{eq1}
d(L(p,q),i)=-\dfrac{1}{4}+\dfrac{(2i+1-p-q)^2}{4pq}-d(L(q,r),j)
\end{equation}
where $r$ and $j$ are the reductions of $p$ and $i$ modulo $q$ respectively.
\end{theorem}

For $L(n,1)$ with $n>0$, this yields the simple formula
\begin{equation}
d(L(n,1),i)=-\dfrac{1}{4}+\dfrac{(2i-n)^2}{4n}.   \label{eq2}
\end{equation}

\subsection{Mapping cone formula for rationally null-homologous knots}
\label{mapping cone for rationally null-homologous knots}
In this section, 
we review the mapping cone formula of Ozsv\'ath and Szab\'o \cite{OSr} for rationally null-homologous knots. %which we employ to derive our $d$-invariant surgery formula. 

Consider a rational homology sphere $Y$ and an oriented knot $K \subset Y$. Let $\mu$ be the meridian of $K$ and $\gamma$ a \textit{framing}, which is an embedded curve on the boundary of the tubular neighborhood of $K$ intersecting $\mu$ transversely once. %Here, $\gamma$ naturally inherits an orientation from $K$. 
The set of relative $\rm Spin^c$ structures over $Y-K$, denoted by $\underline{\rm Spin^c}(Y,K)$, is affinely isomorphic to $H^2(Y,K)$. The natural map defined in \cite[Section 2.2]{OSr},
\[G_{Y, \pm K}: \underline{\rm Spin^c}(Y,K) \rightarrow {\rm Spin^c}(Y) \cong H^2(Y),\]
sends a relative $\rm Spin^c$ structure to a $\rm Spin^c$ structure in the target manifold, satisfying
\[G_{Y,\pm K}(\xi+\kappa)=G_{Y, \pm K}(\xi)+i^{\ast}(\kappa),\]
where $\kappa \in H^2(Y,K)$ and $i^{\ast}: H^2(Y,K) \rightarrow H^2(Y)$ is induced by inclusion. Here, $-K$ denotes $K$ with the opposite orientation. We have
\[G_{Y,-K}(\xi)=G_{Y,K}(\xi)+i^{\ast} PD[\gamma].\]
%where $\gamma$ is the push-off of $K$ inside $Y-K$ using the framing $\gamma$.

For each $\xi \in \underline{\rm Spin^c}(Y,K)$, we associate a $\mathbb{Z} \oplus \mathbb{Z}$-filtered knot Floer complex $C_{\xi}=CFK^{\infty}(Y,K,\xi)$, whose bifiltration is given by $(i,j)=(algebraic, Alexander)$. Let 
\[A^+_{\xi}=C_{\xi}\{\max\{i,j\} \geq 0\} \,\,\, {\rm and} \,\,\, B^+_{\xi}=C_{\xi}\{i \geq 0\}.\]
These complexes are related by grading homogeneous maps
\[v^+_{\xi}: A^+_{\xi} \rightarrow B^+_{\xi}, \quad h^+_{\xi}: A^+_{\xi} \rightarrow B^+_{\xi+PD[\gamma]}\]
that can be identified with certain cobordism maps as follows. Fix $l \gg 0$,  consider the negative definite two-handle cobordism $W'_l$ obtained by turning around the two-handle cobordism from $-Y$ to $-Y_{\gamma+l\mu}(K)$. %Fix a generator $[F] \in H_2(W'_l,Y)$ such that $PD[F] \vert_Y=PD[K]$. 
Given a $\rm Spin^c$ structure $\mathfrak{t}$ on $Y_{\gamma+l\mu}(K)$, Ozsv\'{a}th and Szab\'{o} \cite[Theorem 4.1]{OSr} show the existence of two specific $\rm Spin^c$ structures $\mathfrak{v}$ and $\mathfrak{h}$ on $W'_l$ extending $\mathfrak{t}$ over $W'_l$, and an associated map $\Xi: {\rm Spin^c}(Y_{\gamma+l\mu}(K)) \rightarrow \underline{\rm Spin^c}(Y,K) $ such that the following diagrams commute:
\begin{minipage}{0.5\textwidth}
\begin{equation}
\begin{tikzcd}
 CF^+(Y_{\gamma+l\mu}(K), \mathfrak{t}) \arrow[d,"F_{W'_l,\mathfrak{v}}"] \arrow[r,"\simeq"] &   A^+_{\xi} \arrow[d,"v^+_{\xi}"] \\
 CF^+(Y,G_{Y,K}(\xi)) \arrow[r, "\simeq"]  & B^+_{\xi}
\end{tikzcd}\label{cdeq1}
\end{equation}
\end{minipage}
\begin{minipage}{0.5\textwidth}
\begin{equation}
\begin{tikzcd}
 CF^+(Y_{\gamma+l\mu}(K), \mathfrak{t}) \arrow[d,"F_{W'_l,\mathfrak{h}}"] \arrow[r,"\simeq"] &   A^+_{\xi} \arrow[d,"h^+_{\xi}"] \\
CF^+(Y,G_{Y,-K}(\xi)) \arrow[r, "\simeq"]  & B^+_{\xi+PD[\gamma]},
\end{tikzcd}\label{cdeq2}
\end{equation}
\end{minipage}
where $\xi=\Xi(\mathfrak{t})$, and    $F_{W'_l,\mathfrak{v}}$, $F_{W'_l,\mathfrak{h}}$ are ${\rm Spin}^c$ cobordism maps defined in \cite{OSc}.

Denote
\[\mathfrak{A}^+_{\xi}=H_{\ast}(A^+_{\xi}), \qquad \mathfrak{B}^+_{\xi}=H_{\ast}(B^+_{\xi}).\]
Let
$\mathfrak{v}^+_{\xi}: \mathfrak{A}^+_{\xi} \rightarrow \mathfrak{B}^+_{\xi}$ and $\mathfrak{h}^+_{\xi}: \mathfrak{A}^+_{\xi} \rightarrow \mathfrak{B}^+_{\xi+PD[\gamma]}$
be the maps induced on homology by $v^+_{\xi}$ and $h^+_{\xi}$ respectively. As $HF^+$ of any $\rm Spin^c$ rational homology sphere contains a tower $\mathcal{T}^+$, we define
\[V_{\xi}(K):={\rm rank(ker} \,\, \mathfrak{v}^+_{\xi}|_{\mathcal{T}^+}), \quad H_{\xi}(K):={\rm rank(ker} \,\, \mathfrak{h}^+_{\xi}|_{\mathcal{T}^+}) .\]
Then $\mathfrak{v}^+_{\xi}|_{\mathcal{T}^+}$ and $\mathfrak{h}^+_{\xi}|_{\mathcal{T}^+}$ are multiplications by $U^{V_{\xi}}$ and $U^{H_{\xi}}$ respectively. These sequences of non-negative integers $V_{\xi}(K)$ and $H_{\xi}(K)$ satisfy the monotonicity conditions  
\begin{align}
&V_{\xi}(K) \geq V_{\xi+PD[\mu]}(K) \geq V_{\xi}(K)-1, \label{eq11}\\
&H_{\xi}(K) \leq H_{\xi+PD[\mu]}(K) \leq H_{\xi}(K)+1 \label{01eq}
\end{align}
for each $\xi \in \underline{\rm Spin^c}(Y,K)$.

%Another crucial property of $V_{\xi}(K)$ and $H_{\xi}(K)$ to be repeatedly used is: There is another crucial property of $V_{\xi}(K)$ and $H_{\xi}(K)$, which will be used repeatedly in the later sections. Although it is well-known for experts, for clarity, we will give a proof here.

\begin{lemma}
\label{lemmaV-H}
For two knots $K$ and $K'$ in a rational homology sphere $Y$ with $[K]=[K'] \in H_1(Y)$, and for any relative $\rm Spin^c$ structure $\xi \in \underline{\rm Spin^c}(Y,K)\cong \underline{\rm Spin^c}(Y,K')$,
\[V_{\xi}(K)-H_{\xi}(K)=V_{\xi}(K')-H_{\xi}(K').\]
\end{lemma}

\begin{proof}
As stated above, the two maps $v^+_{\xi}$ and $h^+_{\xi}$ are identified with cobordism maps (\ref{cdeq1}) and (\ref{cdeq2}). By the grading shift formula in \cite[Theorem 7.1]{OSc}, we have
\begin{align}
&d(Y,G_{Y,K}(\xi))-d(Y_{\gamma+l\mu}(K),\mathfrak{t})=\dfrac{c_1(\mathfrak{v})^2-3\sigma(W'_l)-2\chi(W'_l)}{4}-2V_{\xi}(K) \label{02eq}\\
&d(Y,G_{Y,-K}(\xi))-d(Y_{\gamma+l\mu}(K),\mathfrak{t})=\dfrac{c_1(\mathfrak{h})^2-3\sigma(W'_l)-2\chi(W'_l)}{4}-2H_{\xi}(K)
\label{03eq}
\end{align}
Comparing (\ref{02eq}) and (\ref{03eq}), we see that 
$$V_{\xi}(K)-H_{\xi}(K)=\frac{1}{2}d(Y,G_{Y,-K}(\xi))-\frac{1}{2}d(Y,G_{Y,K}(\xi))+\frac {c_1(\mathfrak{v})^2-c_1(\mathfrak{h})^2}{8},$$ 
which is completely determined by homological information. %This proves $V_{\xi}(K)-H_{\xi}(K)=V_{\xi}(K')-H_{\xi}(K')$.
\end{proof}

%As in \cite{OSr}, let $A^+_{\xi}=C_{\xi}\{max\{i,j\} \geq 0\}$, and $B^+_{\xi}=C_{\xi}\{i \geq 0\}$.
%More precisely, for all sufficient large $n$, there is a map
%\[\Xi: {\rm Spin^c}(Y_{\gamma+n\mu}(K)) \rightarrow \underline{\rm Spin^c}(Y,K) \]
%satisfying for all $\mathfrak{t} \in {\rm Spin^c}(Y_{\gamma+n\mu}(K))$, $CF^+(Y_{\gamma+n\mu}(K), \mathfrak{t})=A^+_{\Xi(\mathfrak{t})}$.
%We can view the two projections in Eq. (\ref{eq10}) as maps

%Ozsv\'{a}th and Szab\'{o} show that $v^+_{\xi}$ and $h^+_{\xi}$ correspond to the negative definite cobordism maps $W'_n: Y_{\gamma+n\mu}(Y) \rightarrow Y$ for $n \gg 0$ equipped with certain $\rm Spin^c$ structures. See \cite[Theorem 4.1]{OSr} for details.

\bigskip

For any $\mathfrak{s} \in {\rm Spin^c}(Y_{\gamma}(K))$, we define
\begin{align*}
\mathbb{A^+_{\mathfrak{s}}}&=\mathop{\bigoplus}_{\{\xi \in \underline{\rm Spin^c}(Y_{\gamma}(K),K_{\gamma})| G_{Y_{\gamma}(K),K_{\gamma}}(\xi)=\mathfrak{s}\}} A^+_{\xi}, \\
\mathbb{B^+_{\mathfrak{s}}}&=\mathop{\bigoplus}_{\{\xi \in \underline{\rm Spin^c}(Y_{\gamma}(K),K_{\gamma})| G_{Y_{\gamma}(K),K_{\gamma}}(\xi)=\mathfrak{s}\}} B^+_{\xi},
\end{align*}
where $K_{\gamma}$ represents the oriented dual knot of the knot $K$ in the surgered manifold $Y_{\gamma}(K)$, and $$G_{Y_{\gamma}(K),K_{\gamma}}: \underline{\rm Spin^c}(Y_{\gamma}(K),K_{\gamma}) \rightarrow {\rm Spin^c}(Y_{\gamma}(K)).$$ Note that $\underline{\rm Spin^c}(Y,K)=\underline{\rm Spin^c}(Y_{\gamma}(K),K_{\gamma})$, as both represent the set of the relative ${\rm Spin^c}$ structures on the knot complement $Y-K=Y_{\gamma}(K)-K_{\gamma}$. 

We are now ready to state the mapping cone formula.  Let $\mathbb{X}^+_{\mathfrak{s}}$ denote the mapping cone of the chain map
\[D^+_{\mathfrak{s}}: \mathbb{A}^+_{\mathfrak{s}} \rightarrow \mathbb{B}^+_{\mathfrak{s}}, \quad
(\xi, a) \mapsto (\xi, v^+_{\xi}(a))+(\xi+PD[\gamma], h^+_{\xi}(a))\]
%With this, we are ready to state the connection between the knot Floer complex of the knot $K$ and the Heegaard Floer homology of the manifold obtained from distance one surgery along $K$.
There exist grading shifts on the complexes $\mathbb{A}^+_{\mathfrak{s}}$ and $\mathbb{B}^+_{\mathfrak{s}}$ that give a consistent relative $\mathbb{Z}$-grading on $\mathbb{X}^+_{\mathfrak{s}}$. In \cite[Section 7]{OSr}, Ozsv\'ath and Szab\'o exhibited the canonical grading shifts, depending only on $[K]\in H_1(Y)$, that respect the absolute $\mathbb{Q}$-grading in the following isomorphism.

\begin{theorem}[Ozsv\'ath-Szab\'o \cite{OSr}]
\label{mappingconethm}
Let $\mathbb{X}^+_{\mathfrak{s}}$ be the mapping cone of the chain map $D^+_{\mathfrak{s}}: \mathbb{A}^+_{\mathfrak{s}} \rightarrow \mathbb{B}^+_{\mathfrak{s}}$. Then there exists an absolutely graded isomorphism of groups 
$$H_*(\mathbb{X}^+_{\mathfrak{s}})\cong HF^+(Y_{\gamma}(K), \mathfrak{s}).$$

%For any $\mathfrak{s} \in {\rm Spin^c}(Y_{\gamma}(K))$, the Heegaard Floer homology $HF^+(Y_{\gamma}(K),\mathfrak{s})$ is isomorphic to the homology of the mapping cone $\mathbb{X}^+_{\mathfrak{s}}$ of the chain map $D^+_{\mathfrak{s}}: \mathbb{A}^+_{\mathfrak{s}} \rightarrow \mathbb{B}^+_{\mathfrak{s}}$.

\end{theorem}

\begin{comment}
\begin{remark}
\label{gradingshift}
Indeed, there are unique absolute lifts of the relative $Z$-gradings on both $\mathbb{A}^+_{\mathfrak{s}}$ and $\mathbb{B}^+_{\mathfrak{s}}$ so that the above isomorphism respects the absolute $\mathbb{Q}$-grading. It is important to point out that these grading lifts depend only on the homology class of the knot. See \cite[Section 7]{OSr} for a discussion of the absolute grading.

\end{remark}
\end{comment}

\medskip
\noindent
In practice, one typically works with the homology of the complexes 
\[\mathfrak{A}^+_{\mathfrak{s}}=H_{\ast}(\mathbb{A}^+_{\mathfrak{s}}),  \qquad \mathfrak{B}^+_{\mathfrak{s}}=H_{\ast}(\mathbb{B}^+_{\mathfrak{s}}).\]
Let  
$\mathfrak{D}^+_{\mathfrak{s}}: \mathfrak{A}^+_{\mathfrak{s}} \rightarrow \mathfrak{B}^+_{\mathfrak{s}}$
be the map induced by $D^+_{\mathfrak{s}}$ on homology. This yields an exact triangle: 
\begin{equation*}
\xymatrix{
\mathfrak{A}^+_{\mathfrak{s}} \ar[r]^{\mathfrak{D}^+_{\mathfrak{s}}} & \mathfrak{B}^+_{\mathfrak{s}} \ar[d]^{{incl}_{\ast}} \\
 & %H_{\ast}(\boldsymbol{X}^-_{\mathfrak{s}})=
 HF^+(Y_{\gamma}(K),\mathfrak{s})
 \ar[lu]^{{proj}_{\ast}}.}
\label{mappingconeexacttriangle}
\end{equation*}
Consequently: 
$$HF^+(Y_{\gamma}(K),\mathfrak{s})\cong \ker(\mathfrak{D}^+_{\mathfrak{s}})\oplus \mathrm{coker} (\mathfrak{D}^+_{\mathfrak{s}}).$$

%Thus, $HF^+(Y_{\gamma}(K),\mathfrak{s})$ is isomorphic to the direct sum of the kernel and cokernel of the map $\mathfrak{D}^+_{\mathfrak{s}}$.

%There is an analogous mapping cone formula for the hat version of the Heegaard Floer homology. One can define $\widehat{A}_{\xi}$, $\widehat{B}_{\xi}$, $\widehat{D}_{\xi}$ and the mapping cone $\widehat{\mathbb{X}}_{\mathfrak{s}}$ of $\widehat{D}_{\xi}$, and the Heegaard Floer homology $\widehat{HF}(Y_{\gamma}(K),\mathfrak{s})$ can be calculated by the homology of $\widehat{\mathbb{X}}_{\mathfrak{s}}$.

\subsection{Simple knots in lens spaces}

To compute the $d$-invariant of surgered manifolds $Y_{\gamma}(K)$ using the mapping cone formula, we must first fix the grading shift. For $Y=S^3$, this is done using the unknot $K=U$, setting the bottom grading of $H_*(\mathbb{X}^+_{i, m/n})$ to $d(S^3_{m/n}(U), i)=d(L(m,n),i)$, which can be computed recursively as described in (\ref{thmOS1}). For further details, see \cite[Section 7.2]{OSr}. This approach generalizes naturally to lens spaces $Y=L(p,q)$, where simple knots serve as an analogous role to the unknot in $S^3$.

%As pointed out in Remark \ref{gradingshift}, the grading shift depends only on the homology class of the knot. We will use a simple knot in the same homology class to fix it, since simple knots have the simplest knot Floer complex.
%want to choose a knot in the same homology class with simpler knot Floer complex to fix it. Simple knots in lens spaces will play such a role.

Consider a standard genus one Heegaard diagram for $L(p,q)$ (e.g., $L(5,1)$ in Figure \ref{simpleknot}), represented as a rectangle with opposite sides identified. The horizontal red curve represents the $\alpha$ curve defining a solid torus $U_{\alpha}$, and the blue curve of slope $p/q$ represents the $\beta$ curve defining a solid torus $U_{\beta}$. The two curves intersect at $p$ points, denoted by $x_0, x_1, \dots , x_{p-1}$, which are labeled in the order they appear on the $\alpha$ curve. 

\begin{definition}
\label{Def:simpleknot}
For $k=0,1,\cdots, p-1$, the {\it simple knot} $K(p,q,k) \subset L(p,q)$ is an oriented knot defined as the union of an arc joining $x_k$ to $x_0$ in the meridian disk of $U_{\alpha}$ and an arc joining $x_0$ to $x_k$ in the meridian disk of $U_{\beta}$. 
\end{definition}

\noindent
Clearly, each homology class in $H_1(L(p,q))$ contains exactly one simple knot up to isotopy.

\begin{figure}[!h]
\centering
\includegraphics[width=2.5in]{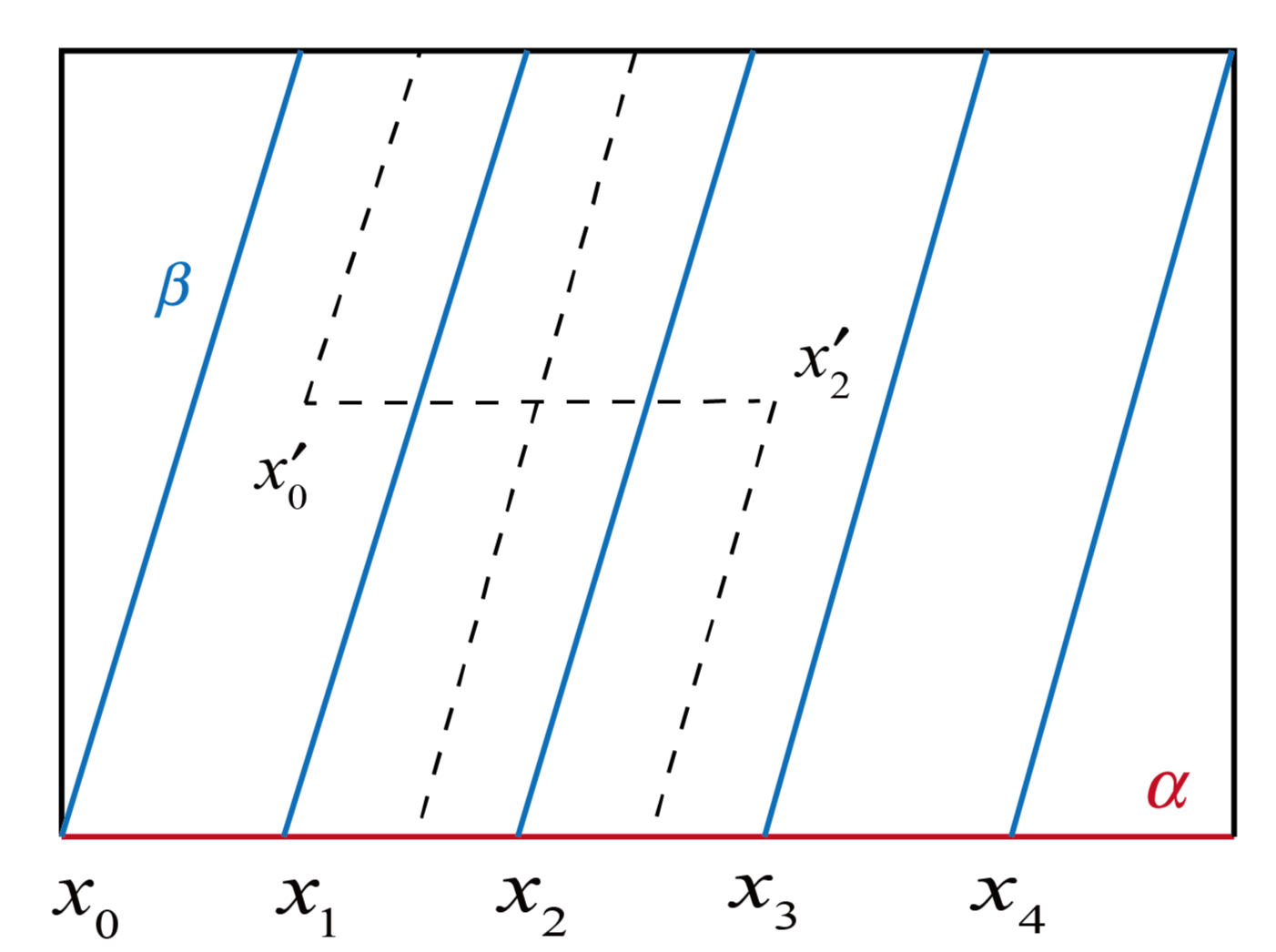}
\caption{An example of a simple knot $K(5,1,2)$ in $L(5,1)$. To draw $K(5,1,2)$ in the Heegaard diagram, we place two points $x'_0$ and $x'_2$ next to $x_0$ and $x_2$ respectively, and connect them in $U_{\alpha}$ and $U_{\beta}$.}
\label{simpleknot}
\end{figure}

\begin{definition}
A knot $K$ in an $L$-space $Y$ is called {\it Floer simple} if $${\rm rank} \,\, \widehat{HFK}(Y,K)={\rm rank} \,\, \widehat{HF}(Y).$$ 
\end{definition}
Specifically, the simple knot $K(p,q,k)$ is Floer simple, as the $p$ intersection points $x_0, \dots , x_{p-1}$ represent generators of $\widehat{CFK}(Y,K)$ in different $\rm Spin^c$ structures.

\subsection{$L$-space knots in $L$-space}

Similar to the case of $S^3$, we call $K\subset Y$ an $L$-space knot if it admits a positive Dehn surgery to an $L$-space:

\begin{definition}\label{Def:L-spacekntos}

In an $L$-space $Y$, a knot $K$ is called an \textit{$L$-space knot} if there exists a positive framing $\gamma$ such that $Y_\gamma(K)$ is an $L$-space. 

\end{definition}

Next, we apply the mapping cone formula to study $Y_\gamma(K)$, where $Y$ is an $L$-space and $\gamma$ is a positive framing. Given $\xi \in \underline{\rm Spin^c}(Y,K)$, the mapping cone $\mathbb{X}^+_{\mathfrak{s}}$ consists of $\mathfrak{A}^+_{\xi+n \cdot PD[\gamma]}$ and $\mathfrak{B}^+_{\xi+n \cdot PD[\gamma]}$ for all $n \in \mathbb{Z}$ %with $\mathfrak{s}=G_{Y_{\gamma}(K), K_{\gamma}}(\xi)$ 
and is depicted in Figure \ref{mappingcone}. %where $\mathfrak{s}=G_{Y_{\gamma}(K), K_{\gamma}}(\xi)$.

\begin{figure}[!ht]
\begin{minipage}{1\linewidth}
\centering{\begin{displaymath}
\xymatrix{
\dots \ar[rd] & \mathfrak{A}^+_{\xi - 2 \cdot PD[\gamma]} \ar[d] \ar[rd] & \mathfrak{A}^+_{\xi - PD[\gamma]} \ar[d] \ar[rd] & \mathfrak{A}^+_{\xi} \ar[d]^{\mathfrak{v}^+_{\xi}} \ar[rd]^{\mathfrak{h}^+_{\xi}} & \mathfrak{A}^+_{\xi + PD[\gamma]} \ar[d] \ar[rd] & \mathfrak{A}^+_{\xi + 2 \cdot PD[\gamma]} \ar[d] \ar[rd] & \dots\\
\dots & \mathfrak{B}^+_{\xi - 2 \cdot PD[\gamma]} &  \mathfrak{B}^+_{\xi - PD[\gamma]} & \mathfrak{B}^+_{\xi} &  \mathfrak{B}^+_{\xi + PD[\gamma]} & \mathfrak{B}^+_{\xi + 2 \cdot PD[\gamma]} & \dots}
\end{displaymath}}
\end{minipage}
\caption{The mapping cone $\mathbb{X}^+_{\mathfrak{s}}$ with $\mathfrak{s}=G_{Y_{\gamma}(K), K_{\gamma}}(\xi)$.} \label{mappingcone}
\end{figure}
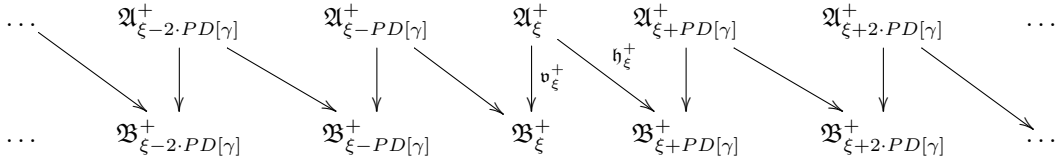

Assuming $Y_{\gamma}(K)$ is an $L$-space, it follows from \cite[Lemma 6.7]{Boil} that
\begin{equation}\label{L-space-knotCriterion}
\mathfrak{A}^+_{\xi}(K) \cong \mathfrak{B}^+_{\xi} \cong \mathcal{T}^+
\end{equation}
for all $\xi \in \underline{\rm Spin^c}(Y,K)$. Conversely, a knot satisfying this condition is an $L$-space knot.  
Under this circumstance, the mapping cone structure is completely determined by the integers $V_{\xi}(K)$ and $H_{\xi}(K)$. This property suggests the following partial order on $L$-space knots.

%Now we define a partial order $\preceq$ on $L$-space knots within a given homology class.
\begin{definition}\label{Def:partialorder}
For an $L$-space $Y$ and two knots $[K]=[K']\in H_1(Y)$, we denote $K' \preceq K$ if
$$V_\xi(K')\leq V_\xi(K)$$ for all $\xi \in \underline{\rm Spin^c}(Y,K)$. By Lemma \ref{lemmaV-H}, this condition is equivalent to  $$H_\xi(K')\leq H_\xi(K).$$ 

\end{definition}

\begin{example}
In $Y=S^3$, the unknot $U$ is an $L$-space knot with $V_i(K)=H_{-i}(K)=0$ for all $i\geq 0$. Consequently, $U$ is minimal under $\preceq$ due to the non-negativity of $V$ and $H$.  More generally, all simple knots in lens spaces are minimal elements in their respective homology classes, since for all $\xi$,  at least one of $V_\xi$ or $H_\xi$ is 0.  

\end{example}

The partial order $K'\preceq K$ provides a measure of the complexity of $L$-space knots, particularly in relation to their Dehn surgeries.  We will prove the following theorem:

\begin{theorem}\label{Thm:L-spaceknotpartialorder}
Suppose $K'\preceq K$ are both $L$-space knots, and a positive surgery $Y_\gamma(K)$ results in an $L$-space.  Then $Y_\gamma(K')$ is also an $L$-space.

\end{theorem}

\subsection{Rasmussen's notation}
\label{Rasmussen's notation}

In this section, we introduce Rasmussen's notation \cite{Ras2, WY2} to provide a simple proof of Theorem \ref{Thm:L-spaceknotpartialorder}. % We will use $\mathfrak{A}_{\xi}$ to represent either $\hat{\mathfrak{A}}_{\xi}$ or $\mathfrak{A}_{\xi}^+$, depending on the context.  

\begin{definition}
Given an $L$-space knot $K$ and $\xi \in \underline{\rm Spin^c}(Y,K)$, we assign one of the following 4 types of symbols to the complex $\mathfrak{A}_\xi$ (where $\mathfrak{A}_\xi$
refers to either $\widehat{\mathfrak{A}}_\xi$ or $\mathfrak{A}_\xi^+$, as both receive the same type assignment)
based on the following criteria:
\begin{itemize}
\item[(1)] $\mathfrak{A}_{\xi}=+$ if $\hat{\mathfrak{v}}_{\xi}$ is non-trivial and $\hat{\mathfrak{h}}_{\xi}$ is trivial;
\item[(2)] $\mathfrak{A}_{\xi}=-$ if  $\hat{\mathfrak{h}}_{\xi}$ is non-trivial and  $\hat{\mathfrak{v}}_{\xi}$ is trivial;
\item[(3)] $\mathfrak{A}_{\xi}=\circ$ if both $\hat{\mathfrak{v}}_{\xi}$ and $\hat{\mathfrak{h}}_{\xi}$ are non-trivial;
\item[(4)] $\mathfrak{A}_{\xi}=\ast$ if both $\hat{\mathfrak{v}}_{\xi}$ and $\hat{\mathfrak{h}}_{\xi}$ are trivial.
\end{itemize}
\end{definition}

Recall that for an $L$-space knot $K$, $\widehat{\mathfrak{A}}_{\xi}(K)=\mathbb{F}$ and $\mathfrak{A}^{+}_{\xi}(K)=\mathcal{T}^+$. Hence, 
\begin{itemize}
\item $V_{\xi}>0 \Longleftrightarrow $ $\hat{\mathfrak{v}}_{\xi}$ is trivial,

\item $V_{\xi}=0 \Longleftrightarrow $ $\hat{\mathfrak{v}}_{\xi}$ is non-trivial,

\item $H_{\xi}>0 \Longleftrightarrow $ $\hat{\mathfrak{h}}_{\xi}$ is trivial,

\item $H_{\xi}=0 \Longleftrightarrow $ $\hat{\mathfrak{h}}_{\xi}$ is non-trivial

\end{itemize}

The above equivalent characterization of Rasmussen's notation leads to the following lemma. 

\begin{lemma}
\label{Lemma:simpleknotrasmussen}
Suppose $K'\preceq K$ are two $L$-space knots and $\xi \in \underline{\rm Spin^c}(Y,K)\cong \underline{\rm Spin^c}(Y,K') $. Then in Rasmussen's notation, whenever $\mathfrak{A}_{\xi}(K)$ is of type $+,-$ or $\circ$, the corresponding $\mathfrak{A}_{\xi}(K')$ must be of the same type. 
\end{lemma}

\begin{proof}
By Definition \ref{Def:partialorder} and Lemma \ref{lemmaV-H}, $V_\xi(K')-V_\xi (K)=H_\xi(K')-H_\xi(K)\leq 0$.  %So if, for instance, $\widehat{\mathfrak{A}}_{\xi}(K)=+$, then $V_\xi(K)=0$ and $H_\xi(K)>0$. Non-negativity of $V$ implies $V_\xi(K')=V_\xi(K)=0$; and consequently $H_\xi(K')=H_\xi(K)>0$.  
This together with the non-negativity of $V$ and $H$ gives the lemma. 

\end{proof}

Rasmussen's notation provides a visual representation of the complex $\widehat{\mathfrak{D}}_{\mathfrak{s}}: \widehat{\mathfrak{A}}_{\mathfrak{s}} \rightarrow \widehat{\mathfrak{B}}_{\mathfrak{s}}$ as shown in Figures \ref{Rasmussen's notation example} and \ref{2nd Rasmussen's notation example}. In these diagrams, the top row represents $\widehat{\mathfrak{A}}_{\xi}$, while the bottom row consists of filled circles representing $\widehat{\mathfrak{B}}_{\xi}$. %Each $\widehat{\mathfrak{B}}_{\xi}$ is denoted by a filled circle. 
Nontrivial maps are indicated by arrows, and trivial maps are omitted. For positive surgery framing $\gamma$, all complexes $\widehat{\mathfrak{A}}_{\xi\pm N\cdot PD[\gamma] }$ must be of types $+$ and $-$, respectively, for sufficiently large $N$. This results in infinitely many $-$ symbols extending to the left and infinitely many $+$ symbols extending to the right; however, these are truncated in our figures as they do not affect the homology of the mapping cone.

\medskip
Our main result in this section gives criteria for $L$-space surgery based on the symbol pattern in Rasmussen's notation.  

\begin{proposition} \label{Prop:L-spacesurgerycriteria}

For a positive framing $\gamma$ of an $L$-space knot $K\subset Y$, the surgered manifold $Y_\gamma(K)$ is an $L$-space if and only if the complex $\widehat{\mathfrak{D}}_{\mathfrak{s}}: \widehat{\mathfrak{A}}_{\mathfrak{s}} \rightarrow \widehat{\mathfrak{B}}_{\mathfrak{s}}$ satisfies the following $3$ conditions:
\begin{enumerate}
\item  There is at most one $\ast$.
\item  All $-$ symbols appear to the left of any $+$ or $\ast$ symbols.
\item All $+$ symbols appear to the right of any $-$ or $\ast$ symbols.
\end{enumerate}

\end{proposition}

As illustrations, the mapping cone in Figure \ref{Rasmussen's notation example} violates (1) and (3), while Figure \ref{2nd Rasmussen's notation example} satisfy all three conditions.

\begin{figure}[!ht]
\begin{minipage}[b]{1\textwidth}
\centering
\begin{displaymath}
\xymatrixrowsep{5mm}
\xymatrixcolsep{4mm}
\xymatrix{
- \ar[rd] & - \ar[rd] & \circ \ar[d] \ar[rd] & + \ar[d] & \circ \ar[d] \ar[rd] & \ast & \ast & \circ \ar[d] \ar[rd] & + \ar[d] &  + \ar[d]\\
 & \bullet & \bullet & \bullet  & \bullet & \bullet & \bullet & \bullet & \bullet  & \bullet}
\end{displaymath}
\end{minipage}
\caption{A truncated mapping cone in Rasmussen's notation. While the full complex extends infinitely in both directions, we show only the essential portion that determines the homology.
}
\label{Rasmussen's notation example}
\end{figure}
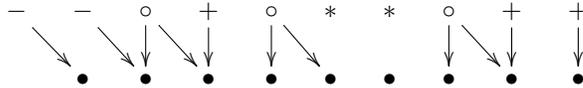

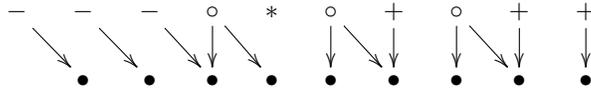
\begin{figure}[!ht]
\begin{minipage}[b]{1\textwidth}
\centering
\begin{displaymath}
\xymatrixrowsep{5mm}
\xymatrixcolsep{4mm}
\xymatrix{
 & - \ar[rd] & - \ar[rd] & - \ar[rd] & \circ \ar[d] \ar[rd]  & \ast & \circ \ar[d] \ar[rd] & + \ar[d] & \circ \ar[d] \ar[rd] & + \ar[d] & + \ar[d]    \\
 & & \bullet  & \bullet & \bullet & \bullet & \bullet & \bullet & \bullet & \bullet & \bullet   }
\end{displaymath}
\end{minipage}
\caption{A mapping cone satisfying conditions in Proposition \ref{Prop:L-spacesurgerycriteria}, shown in truncated form. The complex extends infinitely with $-$ symbols to the left and $+$ symbols to the right.} %We assume all labels on the left are $-$, while all labels on the right are $+$, both of which are omitted from the diagram.}
\label{2nd Rasmussen's notation example}
\end{figure}

\begin{proof}
The proof is a minor adaptation of the argument in \cite[Section 4.3]{Ras2} and is outlined as follows: We apply the mapping cone formula $$\widehat{HF}(Y_{\gamma}(K),\mathfrak{s})\cong \ker(\widehat{\mathfrak{D}}_{\mathfrak{s}})\oplus \mathrm{coker} (\widehat{\mathfrak{D}}_{\mathfrak{s}})$$ and use the fact that ${\rm rank}\, \widehat{HF}(Y_{\gamma}(K),\mathfrak{s})=1$ for $L$-spaces. 

First, each $\ast$ contributes a generator in $\ker(\widehat{\mathfrak{D}}_{\mathfrak{s}})$, so we cannot have more than one $\ast$ for an $L$-space surgery, which establishes Condition (1). 

Next,  noting that $\ker(\widehat{\mathfrak{D}}_{\mathfrak{s}})$ has rank at least 1 for any positive surgery mapping cone, a well-known property to the expert due to the infinite extension of $-$ symbols to the left and $+$ symbols to the right, we observe that certain intervals in the complex contribute additional generators to $\mathrm{coker} (\widehat{\mathfrak{D}}_{\mathfrak{s}})$. We define an interval $[a,b]$ as a local segment of the mapping cone where $a$ and $b$ are symbols of type $+$, $-$, or $\ast$, with only $\circ$ symbols (if any) appearing between them (see Figure \ref{bottom row}). In particular, intervals of type $[+,-]$, $[+, \ast]$, or $[\ast, -]$ each contribute a generator to $\mathrm{coker} (\widehat{\mathfrak{D}}_{\mathfrak{s}})$. This leads to Conditions (2) and (3).
%we observe that an interval of types $[+,-], [+, \ast]$ or $[\ast, -]$ contributes a generator in $\mathrm{coker} (\widehat{\mathfrak{D}}_{\mathfrak{s}})$.  An interval $[a,b]$ represents a local segment of the mapping cone, where $a$ and $b$ are symbols of type $+$, $-$ or $\ast$, with only $\circ$ symbols (if any) appearing between them (See Figure \ref{bottom row}). These lead to Conditions (2) and (3).  

Conversely, suppose a complex $\widehat{\mathfrak{D}}_{\mathfrak{s}}$ satisfies all three conditions. If there is exactly one $\ast$, then all $-$ symbols must appear to the left of the $\ast$ and all $+$ symbols to the right; If there is no $\ast$, then there is a unique $[-,+]$ interval, as illustrated in Figure \ref{top row}. In either case, we find that $\mathrm{coker} (\widehat{\mathfrak{D}}_{\mathfrak{s}})=0$, while $\ker(\widehat{\mathfrak{D}}_{\mathfrak{s}})$ has rank $1$, generated by either the $\ast$ or the $[-,+]$ interval, respectively. Thus, ${\rm rank}\, \widehat{HF}(Y_{\gamma}(K),\mathfrak{s})=1$, confirming that $Y_{\gamma}(K)$ is an $L$-space.

\begin{figure}[H]%htbp
	\centering
	\subfigure[Interval $\lbrack +,- \rbrack$]{
		\begin{minipage}[t]{0.25\textwidth}
			\centering
			\begin{displaymath}
			\xymatrixrowsep{5mm}
			\xymatrixcolsep{4mm}
			\xymatrix{
				+ \ar[d] & \circ \ar[d] \ar[rd] & - \ar[rd]\\
				\bullet & \bullet & \bullet& \bullet}
			\end{displaymath}
		\end{minipage}}
		\subfigure[Interval $\lbrack +, \ast \rbrack$]{
			\begin{minipage}[t]{0.25\textwidth}
				\centering
				\begin{displaymath}
				\xymatrixrowsep{5mm}
				\xymatrixcolsep{4mm}
				\xymatrix{
					+ \ar[d] & \circ \ar[d] \ar[rd] & \ast  \\
					\bullet & \bullet & \bullet }
				\end{displaymath}
			\end{minipage}}
			\subfigure[Interval $\lbrack \ast, - \rbrack$]{
				\begin{minipage}[t]{0.25\textwidth}
					\centering
					\begin{displaymath}
					\xymatrixrowsep{5mm}
					\xymatrixcolsep{4mm}
					\xymatrix{
						\ast  & \circ \ar[d] \ar[rd] & - \ar[rd] &\\
						& \bullet & \bullet & \bullet}
					\end{displaymath}
				\end{minipage}}
				%\subfigure[Summand $\lbrack \ast,\ast \rbrack$]{
				%	\begin{minipage}[t]{0.2\textwidth}
				%		\centering
				%		\begin{displaymath}
				%		\xymatrixrowsep{5mm}
				%		\xymatrixcolsep{4mm}
			%			\xymatrix{
		%					\ast & \circ \ar[d] ar[rd] & \ast\\
			%				& \bullet & \bullet }
				%		\end{displaymath}
					%\end{minipage}}
\caption{Three types of intervals that contribute to $\mathrm{coker} (\widehat{\mathfrak{D}}_{\mathfrak{s}})$: $[+,-]$, $[+, \ast]$, and $[\ast, -]$. Each interval has homology $\mathbb{F}$ supported by an element in the bottom row. For simplicity, we show one $\circ$ between the endpoints in each diagram, though there may be multiple or no $\circ$ symbols in actual cases.}\label{bottom row}
\end{figure}
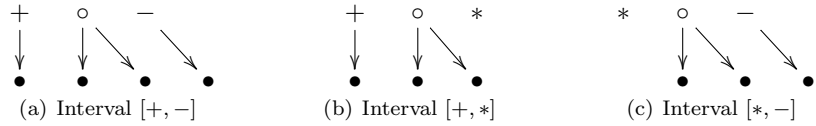

\begin{comment}

\begin{figure}[H]
	\centering
	%\subfigure[Interval $\lbrack -,+ \rbrack$]{
		%\begin{minipage}[t]{0.2\textwidth}
			\centering
\begin{displaymath}
			\xymatrixrowsep{5mm}
			\xymatrixcolsep{4mm}
			\xymatrix{
				- \ar[rd] & \circ \ar[d] \ar[rd] & + \ar[d]\\
			   & \bullet& \bullet}
			\end{displaymath}
\caption{Interval $[-,+]$ that has homology $\mathbb{F}$, generated by the sum of all elements in the top row.} \label{top row}
\end{figure}

\end{comment}

\begin{figure}[H]
	\centering
	%\subfigure[Interval $\lbrack -,+ \rbrack$]{
		%\begin{minipage}[t]{0.2\textwidth}
			\centering
			\begin{displaymath}
			\xymatrixrowsep{5mm}
			\xymatrixcolsep{4mm}
		\xymatrix{
			\textcolor{lightgray}{-}  \ar@[lightgray][rd] &  \textcolor{lightgray} - \ar@[lightgray][rd]&    \textcolor{lightgray}\circ \ar@[lightgray][d]\ar@[lightgray][rd]&  \textcolor{lightgray} - \ar@[lightgray][rd]&- \ar[rd] & \circ \ar[d] \ar[rd] & + \ar[d] & \textcolor{lightgray} + \ar@[lightgray][d] & \textcolor{lightgray} \circ \ar@[lightgray][d] \ar@[lightgray][rd]&  \textcolor{lightgray}\circ \ar@[lightgray][d] \ar@[lightgray][rd] & \textcolor{lightgray}  + \ar@[lightgray][d] \\
		& \textcolor{lightgray}{\bullet}& \textcolor{lightgray}{\bullet}&
            \textcolor{lightgray}{\bullet} & \textcolor{lightgray}{\bullet} & \bullet& \bullet&
            \textcolor{lightgray}{\bullet}& \textcolor{lightgray}{\bullet}& \textcolor{lightgray}{\bullet}& \textcolor{lightgray}{\bullet}}
			\end{displaymath}
		%\end{minipage}}
		%\subfigure[Summand $\lbrack \ast \rbrack$]{
		%	\begin{minipage}[t]{0.2\textwidth}
		%		\centering
		%		\begin{displaymath}
		%		\xymatrixrowsep{5mm}
		%		\xymatrixcolsep{4mm}
		%		\xymatrix@H=4pt{
		%			\ast  & \\
		%		           &}
		%		\end{displaymath}
		%	\end{minipage}}
\caption{A mapping cone without $\ast$ symbols where all $-$ symbols appear to the left of all $+$ symbols. The complex contains a unique $[-,+]$ interval (shown in black), while the grayed-out portions extend infinitely in both directions.  The homology of this mapping cone is $\mathbb{F}$, generated by the sum of all elements in the top row of the interval.} \label{top row}
\end{figure}
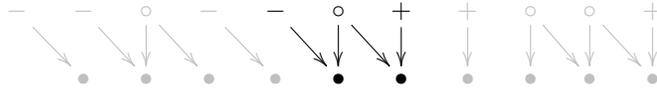

\end{proof}

\begin{proof}[Proof of Theorem \ref{Thm:L-spaceknotpartialorder}]
By Lemma \ref{Lemma:simpleknotrasmussen}, the symbol pattern for $K'$ preserves all $+, -$, and $\circ$ symbols from $K$. Since $Y_\gamma(K)$ is an $L$-space, its symbol pattern satisfies the conditions of Proposition \ref{Prop:L-spacesurgerycriteria}. Therefore, the pattern for $K'$ must also satisfy these conditions, implying $Y_\gamma(K')$ is an $L$-space.

\end{proof}

\section{$d$-invariant surgery formula}
\label{d-invariant surgery formula for homologically essential knots}

Let $K$ be an $L$-space knot in $Y$, and let $\gamma$ be a positive framing such that $Y_\gamma(K)$ is also an $L$-space.  We will analyze the mapping cone $\mathbb{X}^+_{\mathfrak{s}}$ shown in Figure \ref{mappingcone} to determine $d(Y_\gamma(K), \mathfrak{s})$.

\subsection{$L$-space surgery formula}

\begin{definition}\label{Def:GradingShift}
For a framing $\gamma$ and $\xi \in \underline{\rm Spin^c}(Y,K)$, %Define a \textit{grading shift term} 
we define the \textit{grading shift term} $N(K, \gamma, \xi)$ associated with the surgery $Y_\gamma(K)$ as:
\begin{equation}
N(K, \gamma, \xi) := \max_{n \in \mathbb{Z}}\{\min\{V_{\xi+n \cdot PD[\gamma]}(K),H_{\xi+n \cdot PD[\gamma]}(K) \} \}.
\end{equation}
\end{definition}

In Rasmussen's notation, $\min\{V_{\xi+n\cdot PD[\gamma]}(K),H_{\xi+n\cdot PD[\gamma]}(K) \}=0$ if and only if the corresponding complex $\mathfrak{A}_{\xi+n\cdot PD[\gamma]}(K)=+,-$ or $\circ$. Furthermore, Proposition \ref{Prop:L-spacesurgerycriteria} implies that for an $L$-space surgery, at most one $\mathfrak{A}_{\xi+n\cdot PD[\gamma]}(K)$ can be of type $\ast$. This leads to:

\begin{proposition}\label{Prop:computeN}
For a knot $K$ in an $L$-space $Y$ with positive framing $\gamma$ where $Y_\gamma(K)$ is an $L$-space,
\[ N(K, \gamma,\xi)= \begin{cases}
0 & \text{if no } \mathfrak{A}_{\xi+n\cdot PD[\gamma]}(K)  \text { is }\ast\\
\min \{V_{\xi+n_0\cdot PD[\gamma] }(K),H_{\xi+n_0 \cdot PD[\gamma]}(K) \} & \text{if } \mathfrak{A}_{\xi+n_0\cdot PD[\gamma]}(K)=\ast \,\, \text{for some} \,\, n_0. 
\end{cases} \]
\end{proposition}

\medskip
\noindent
Next, we derive a $d$-invariant formula for $L$-space surgeries $Y_\gamma(K)$. This statement can be applied to a pair of \textit{comparable} knots 
$K'\preceq K$ under the partial order defined in Definition \ref{Def:partialorder}.

\begin{theorem}\label{Thm:L-spaceSurgeryFormula}
Suppose $K'\preceq K$ are both $L$-space knots, and $\gamma$ is a positive framing such that $Y_\gamma(K)$ is an $L$-space.  Let $\xi \in \underline{\rm Spin^c}(Y,K)\cong \underline{\rm Spin^c}(Y,K')$ and $\mathfrak{s}=G_{Y_{\gamma}(K), K_{\gamma}}(\xi)=G_{Y_{\gamma}(K'), K'_{\gamma}}(\xi).$ %\in {\rm Spin^c}(Y_\gamma(K))\cong {\rm Spin^c}(Y_\gamma(K')).$$ 
Then, we have 
\begin{equation}\label{eq:L-spacesurgeryd-inv}
d(Y_{\gamma}(K),\mathfrak{s})+2N(K, \gamma, \xi)=d(Y_{\gamma}(K'), \mathfrak{s})+2N(K', \gamma, \xi).
\end{equation}
As a result, 
\begin{equation}\label{Ineq:d-invpartialorder}
    d(Y_{\gamma}(K),\mathfrak{s})\leq d(Y_{\gamma}(K'), \mathfrak{s}).
\end{equation}
    
\end{theorem}

\medskip

\begin{proof}%[Proof of Theorem \ref{d-inv surgery formula for all spinc}]
For any relative $\rm Spin^c$ structure $\xi$, we will use the mapping cone $\mathfrak{D}^+_{\mathfrak{s}}: \mathfrak{A}^+_{\mathfrak{s}} \rightarrow \mathfrak{B}^+_{\mathfrak{s}}$ to prove the $d$-invariant formula. %, where $\mathfrak{s}=G_{Y_{\gamma}(K), K_{\gamma}}(\xi)=G_{Y_{\gamma}(K'), K'_{\gamma}}(\xi)$.  
By Theorem \ref{Thm:L-spaceknotpartialorder}, $Y_\gamma(K')$ is also an $L$-space.  As in proof of other versions of $d$-invariant surgery formulas, the key step is identifying where the nonzero element of minimal grading in 
$$HF^+(Y_{\gamma}(K),\mathfrak{s})=HF^+(Y_{\gamma}(K'),\mathfrak{s})=\mathcal{T}^+$$ 
is supported in $\mathfrak{D}^+_{\mathfrak{s}}: \mathfrak{A}^+_{\mathfrak{s}} \rightarrow \mathfrak{B}^+_{\mathfrak{s}}$. 

From earlier discussions, we know there is at most one $\ast$ among $\mathfrak{A}_{\xi+n\cdot PD[\gamma]}(K)$ (and $\mathfrak{A}_{\xi+n\cdot PD[\gamma]}(K')$).
We consider two cases.

 \medskip   \noindent
\textit{Case 1}: There is no $\ast$ among $\mathfrak{A}_{\xi+n\cdot PD[\gamma]}(K)$. In this case, all $\mathfrak{A}_{\xi+n\cdot PD[\gamma]}(K)$ are of type $+, -$ or $\circ$.  By Lemma \ref{Lemma:simpleknotrasmussen}, the same holds for $\mathfrak{A}_{\xi+n\cdot PD[\gamma]}(K')$. Specifically, by Lemma \ref{lemmaV-H} we have 
$$V_{\xi+n\cdot PD[\gamma]}(K)=V_{\xi+n\cdot PD[\gamma]}(K') \quad \text{and} \quad H_{\xi+n\cdot PD[\gamma]}(K)=H_{\xi+n\cdot PD[\gamma]}(K')$$
for all $n$. This yields isomorphic mapping cones $\mathfrak{D}^+_{\mathfrak{s}}: \mathfrak{A}^+_{\mathfrak{s}} \rightarrow \mathfrak{B}^+_{\mathfrak{s}}$ for $K$ and $K'$, leading to 
$$d(Y_{\gamma}(K),\mathfrak{s})=d(Y_{\gamma}(K'),\mathfrak{s}).$$
As Proposition \ref{Prop:computeN} gives 
$$N(K, \gamma,\xi)=N(K', \gamma, \xi)=0,$$
we obtain formula (\ref{eq:L-spacesurgeryd-inv}).

\medskip
\noindent
\textit{Case 2}: There is exactly one $\ast$ at $\mathfrak{A}_{\xi+n_0\cdot PD[\gamma]}(K)$.  In this case, all other $\mathfrak{A}_{\xi+n\cdot PD[\gamma]}(K)$ are of types $+, -$ or $\circ$.  Again, by Lemma \ref{Lemma:simpleknotrasmussen}, the same types hold for corresponding $\mathfrak{A}_{\xi+n\cdot PD[\gamma]}(K')$, except possibly for $\mathfrak{A}_{\xi+n_0\cdot PD[\gamma]}(K')$, which may or may not be $\ast$.  

\medskip
\noindent \textit{Subcase 2a}:
If $\mathfrak{A}_{\xi+n_0\cdot PD[\gamma]}(K')$ is also of type $\ast$, like its counterpart $\mathfrak{A}_{\xi+n_0\cdot PD[\gamma]}(K)$, then this $\ast$ generates $\ker(\widehat{\mathfrak{D}}_{\mathfrak{s}})$. Since both resulting manifolds are $L$-spaces, we can conclude that the correction terms %being the lowest degree of the tower $\mathcal{T}^+$, 
are determined by $$\mathfrak{v}_{\xi+n_0\cdot PD[\gamma]}^+= U^{V_{\xi+n_0\cdot PD[\gamma]}}:  \mathfrak{A}^+_{\xi+n_0\cdot PD[\gamma]} \rightarrow  \mathfrak{B}^+_{\xi+n_0\cdot PD[\gamma]},$$
which has a fixed grading shift (independent of $K$ or $K'$) given by
\begin{equation}\label{eq:gradingshift}
d(Y_\gamma(K),\mathfrak{s})+2V_{\xi+n_0\cdot PD[\gamma]}(K)-\mathrm{gr}(1)=d(Y_\gamma(K'),\mathfrak{s})+2V_{\xi+n_0\cdot PD[\gamma]}(K')-\mathrm{gr}(1),
\end{equation}
where $\mathrm{gr}(1)$ denotes the grading of the generator $1\in \mathfrak{B}^+_{\xi+n_0\cdot PD[\gamma]}$, independent of the knot.  
Formula (\ref{eq:L-spacesurgeryd-inv}) thus reduces to
\begin{align*}
N(K, \gamma, \xi)-N(K', \gamma,\xi) &= V_{\xi+n_0\cdot PD[\gamma]}(K)-V_{\xi+n_0\cdot PD[\gamma]}(K'),  %\\  &=H_{\xi+n_0\cdot PD[\gamma]}(K)-H_{\xi+n_0\cdot PD[\gamma]}(K').
\end{align*}
which follows from Lemma \ref{lemmaV-H} and Proposition \ref{Prop:computeN}.

\medskip
\noindent \textit{Subcase 2b}:
If $\mathfrak{A}_{\xi+n_0\cdot PD[\gamma]}(K')$ is not of type $\ast$ but rather type $+, - $ or $\circ$, its type is determined by the sign of $$V_{\xi+n_0\cdot PD[\gamma]}(K')-H_{\xi+n_0\cdot PD[\gamma]}(K'),$$ which equals $V_{\xi+n_0\cdot PD[\gamma]}(K)-H_{\xi+n_0\cdot PD[\gamma]}(K).$  Importantly, regardless of the type of $\mathfrak{A}_{\xi+n_0\cdot PD[\gamma]}(K')$, all $-$ symbols will appear to the left of all $+$ symbols, and the unique $[-,+]$ interval generating $\ker(\widehat{\mathfrak{D}}_{\mathfrak{s}})$ contains the complex $\widehat{\mathfrak{A}}_{\xi+n_0\cdot PD[\gamma]}(K')$.
See Figure \ref{Fig:mappingconeKvsK'} for illustration.

\begin{figure}[!h] 
\begin{minipage}[b]{1\textwidth}
\centering
\begin{displaymath}
\xymatrixrowsep{5mm}
\xymatrixcolsep{4mm}
\xymatrix{
\widehat{\mathfrak{A}}_{\xi+n\cdot PD[\gamma]}(K): & \dots & \circ & - & - & \circ & \bm{\ast} & +   & \circ  & +  & \dots\\
\widehat{\mathfrak{A}}_{\xi+n\cdot PD[\gamma]}(K'): & \dots  & \circ & - & - & \circ & -/ \circ /+ & +   & \circ  & +  & \dots}
\end{displaymath}
\end{minipage}
\caption{Comparison of complexes $\widehat{\mathfrak{A}}_{\xi+n\cdot PD[\gamma]}(K)$ and $\widehat{\mathfrak{A}}_{\xi+n\cdot PD[\gamma]}(K')$ in Rasmussen's notation.  Observe that regardless of whether $\widehat{\mathfrak{A}}_{\xi+n_0\cdot PD[\gamma]}(K')$ is of type $+, - $ or $\circ$, all $-$ symbols appear to the left of all $+$ symbols.} 
\label{Fig:mappingconeKvsK'}
\end{figure}

The key observation is that $\ker(\widehat{\mathfrak{D}}_{\mathfrak{s}})$ is a homogeneous element that can be expressed as the sum of all generators in $\widehat{\mathfrak{A}}_\xi(K')$ within the $[-,+]$ interval. % (depicted in Figure \ref{top row}).  
Therefore, the minimal-grading nonzero element in $HF^+(Y_{\gamma}(K'),\mathfrak{s})=\mathcal{T}^+$ is also supported at $\mathfrak{A}^+_{\xi+n_0\cdot PD[\gamma]}(K')$. Consequently, the grading shift (\ref{eq:gradingshift}) holds, and the proof follows the same reasoning as in the previous case.  

Note that $N(K, \gamma, \xi)\geq N(K', \gamma, \xi)$ follows directly from Proposition \ref{Prop:computeN} and the definition of $K'\preceq K$, thus establishing (\ref{Ineq:d-invpartialorder}) and completing the proof.

\end{proof}

As a special case, recall that a Floer simple knot in an  $L$-space is minimal in its homology class since $\min \{V_\xi, H_\xi\}=0$ for all relative $\rm Spin^c$ structures $\xi$. Thus, % Theorem \ref{Thm:L-spaceSurgeryFormula} then implies:

\begin{corollary} \label{Cor:d-inv surgery formula1}
For a knot $K$ in an $L$-space $Y$ with positive framing $\gamma$ where $Y_\gamma(K)$ is an $L$-space, if $K'$ is Floer simple with $[K']=[K] \in H_1(Y)$, then for any $\xi \in \underline{\rm Spin^c}(Y,K)\cong \underline{\rm Spin^c}(Y,K')$
\begin{equation}
\label{eq:L-spacesurgery}
d(Y_{\gamma}(K),\mathfrak{s})=d(Y_{\gamma}(K'),\mathfrak{s})-2N(K, \gamma, \xi),
\end{equation}
where $\mathfrak{s}=G_{Y_{\gamma}(K), K_{\gamma}}(\xi)=G_{Y_{\gamma}(K'), K'_{\gamma}}(\xi)$.
\end{corollary}

\begin{example}
Consider a positive $p$-surgery on a knot in $S^3$. %With the standard identification of the relative $\rm Spin^c$ structure $\xi$ by an integer $i\in \underline{\rm Spin^c}(Y,K) \cong \mathbb{Z}$, we have $V_i(K)-H_i(K)=-i$ for all knot $K$. Thus, if we have a positive $p$-surgery on $K$, then 
The grading shift term is
$$N(K, p\mu+\lambda,i)=\max_{n \in \mathbb{Z}}\{\min\{V_{i+np}(K),H_{i+np}(K) \} \}$$ 
under the standard identification $\xi=i\in \underline{\rm Spin^c}(Y,K) \cong \mathbb{Z}$.  In particular, when $0\leq i \leq p-1$, monotonicity of $V$ and $H$ implies $$N(K, p\mu+\lambda,i)=\max\{ V_{i}(K), H_{i-p}(K)\}.$$
%If we let $0\leq [i]\leq p-1$ be the value of $i$ mod $p$, then by monotonicity it is easy to see that $$N(S^3_p(K),i)=\max\{ V_{[i]}(K), H_{[i]-p}(K)\}.$$  
By substituting $K'=U$ the unknot into (\ref{eq:L-spacesurgeryd-inv}), we obtain 
$$d(S^3_p(K), i)=d(L(p,1),i)-2 \max\{ V_{i}(K), H_{i-p}(K)\}, $$ 
which recovers the well-known $d$-invariant surgery formula, known to hold even without the $L$-space assumption on $S^3_p(K)$ \cite[Proposition 1.6]{NiWu}.    
\end{example}

\subsection{Lens space surgeries}

For the remainder of this section, we focus on surgeries on a lens space $Y=L(n,1)$.  Consider a homologically essential knot $K$ with winding number $k$, where surgery along $K$ with slope $\gamma=(m\mu+\lambda)$ yields another lens space $Y_{\gamma}(K)=L(s,1)$.  Let $K'=K(n,1,k)$ be the simple knot from Definition \ref{Def:simpleknot}.  Then $[K']=[K]\in H_1(Y)$ and $Y_\gamma(K')$ is a Seifert fibered space (cf. \cite[Lemma 9]{Darcysumner}).  In Hatcher's notation \cite[Chapter 2.1]{Hat}, this space can be expressed as  
\begin{equation}
\label{simpleknotSFS}    
M=M\left(0,0;\frac{1}{n-k},\frac{1}{k},\frac{1}{m-k}\right),
\end{equation}
where the two $0$'s indicate that the base space for $M$ is $S^2$ (genus $0$, without boundary), and the fractions $1/a$ with $a=n-k, k$ or $m-k$ specify the types of possible exceptional fibers. 
While Theorem \ref{Thm:L-spaceknotpartialorder} implies that $M$ must be an $L$-space, Corollary \ref{Cor:d-inv surgery formula1} provides additional constraints in terms of $d$-invariants.

\begin{proposition}\label{prop rational d-inv1}
Suppose $Y=L(n,1)$ is a lens space and $K \subset Y$ is a homologically essential knot with winding number $k$.  If there is a positive framing $\gamma=(m\mu+\lambda)$ such that $Y_\gamma(K)=L(s,1)$, then for the simple knot $K'=K(n,1,k)$, the corresponding surgered manifold $Y_\gamma(K')=M$ is a Seifert fibered $L$-space satisfying the following: for any self-conjugate ${\rm Spin^c}$ structure $\mathfrak{t} \in {\rm Spin^c}(L(s,1))$, there exists a non-negative integer $N_0$ and a self-conjugate ${\rm Spin^c}$ structure $\mathfrak{t}^M$ on $M$ such that 
\begin{equation}
\label{dinv1}
d(L(s,1),\mathfrak{t})=d(M, \mathfrak{t}^M)-2N_0.
\end{equation}
Furthermore, when $N_0 \geq 2$, there exists an integer $N_1$ with $N_0-1 \leq N_1 \leq N_0$ satisfying
\begin{equation}
\label{dinv2}
d(L(s,1), \mathfrak{t} \pm i^{\ast}PD[\mu])=d(M, \mathfrak{t}^M \pm i^{\ast}PD[\mu])-2N_1,
\end{equation}
where $i$ denotes the inclusion $Y-K \rightarrow Y_{\gamma}(K)$ or $Y-K' \rightarrow M$.

\end{proposition}

\begin{remark}
Under the identification ${\rm Spin^c}(L(p,q)) \cong \mathbb{Z}_p$ in Theorem \ref{thmOS1}, the \textit{self-conjugate} $\rm Spin^c$ structures on $L(p,q)$ correspond to the set 
\[ \mathbb{Z} \cap \left\{\dfrac{p+q-1}{2}, \dfrac{q-1}{2}\right\}.\]
Thus, for the lens space $L(s,1)$, there is a unique self-conjugate $\rm Spin^c$ structure corresponding to $0$ when $s$ is odd, and there are two self-conjugate $\rm Spin^c$ structures corresponding to $0$ and $\frac{s}{2}$ when $s$ is even.

\end{remark}

\begin{proof}
For each self-conjugate ${\rm Spin^c}$ structure $\mathfrak{t}$ on $Y_{\gamma}(K)=L(s,1)$, there exists a relative ${\rm Spin^c}$ structure $\xi$ such that $G_{Y_{\gamma}(K), K_{\gamma}}(\xi)=\mathfrak{t}$.  The corresponding ${\rm Spin^c}$ structure $\mathfrak{t}^M:=G_{Y_{\gamma}(K'),K'_{\gamma}}(\xi)$ on $Y_\gamma(K')=M$ is also self-conjugate. Applying Corollary \ref{Cor:d-inv surgery formula1}, we obtain
\begin{equation*}
d(L(s,1), \mathfrak{t})=d(Y_{\gamma}(K),\mathfrak{t})=d(M, \mathfrak{t}^M)-2N(K, \gamma, \xi),
\end{equation*}
which gives (\ref{dinv1}).

Now, let us consider the case where $N_0 := N(K, \gamma, \xi)\geq 2$. By Proposition \ref{Prop:computeN}, there exists a relative ${\rm Spin^c}$ structure $\xi_0=\xi+n_0\cdot PD[\gamma]$ with 
$$N(K, \gamma, \xi)=\min \{V_{\xi_0 }(K),H_{\xi_0}(K) \} \geq 2.  $$
Consider first the case where $H_{\xi_0}(K) \geq V_{\xi_0}(K) \geq 2$. 
From (\ref{eq11}) and (\ref{01eq}), we obtain
\[H_{\xi_0+PD[\mu]}(K) \geq H_{\xi_0}(K) \geq  V_{\xi_0}(K)  \geq V_{\xi_0+PD[\mu]} \geq V_{\xi_0}(K)-1\geq 1.\]
This implies $\mathfrak{A}_{\xi_0+PD[\mu]}$ is of type $\ast$ in Rasmussen's notation. %By the uniqueness of the $\ast$ label for $L$-space surgeries (Proposition \ref{Prop:L-spacesurgerycriteria}) and 
Applying Proposition \ref{Prop:computeN}, we see
$$N(K, \gamma, \xi+PD[\mu])=\min  \{V_{\xi_0+PD[\mu]}(K),H_{\xi_0+PD[\mu]}(K) \}.$$ 
Setting $N_1 = N(K, \gamma, \xi+PD[\mu])$ and applying
Corollary \ref{Cor:d-inv surgery formula1} to $\xi+PD[\mu]$  gives
$$d(L(s,1), \mathfrak{t} + i^{\ast}PD[\mu])=d(M, \mathfrak{t}^M + i^{\ast}PD[\mu])-2N_1 $$ 
where $N_0-1\leq N_1 \leq N_0$. 

To prove the identity with the same $N_1$ for $\mathfrak{t} - i^{\ast}PD[\mu]$, we employ conjugate invariance. 
Since $\mathfrak{t}$ is self-conjugate,
$$J(\mathfrak{t}+i^{\ast}PD[\mu])=J\mathfrak{t}-i^{\ast}PD[\mu]=\mathfrak{t}-i^{\ast}PD[\mu],$$ which implies
$$d(L(s,1), \mathfrak{t} + i^{\ast}PD[\mu])   =  d(L(s,1), \mathfrak{t} -i^{\ast}PD[\mu]).$$
The same conjugate relation holds for $M$, giving
$$d(M, \mathfrak{t}^M + i^{\ast}PD[\mu])=d(M, \mathfrak{t}^M - i^{\ast}PD[\mu])$$
Therefore, we conclude $$d(L(s,1), \mathfrak{t} - i^{\ast}PD[\mu])=d(M, \mathfrak{t}^M - i^{\ast}PD[\mu])-2N_1.$$

\medskip

For the case $V_{\xi_0}(K) \geq H_{\xi_0}(K) \geq 2$, the argument proceeds symmetrically, beginning with $\xi_0-PD[\mu]$ and extending to $\xi_0+PD[\mu]$ via conjugate symmetry.  This completes the proof.

\begin{comment}

Let us first study the case where $H_{\xi_0}(K) \geq V_{\xi_0}(K) \geq 2$. 
From (\ref{eq11}) and (\ref{01eq}), we have
\[H_{\xi_0+PD[\mu]}(K) \geq H_{\xi_0}(K) \geq  V_{\xi_0}(K)  \geq V_{\xi_0+PD[\mu]} \geq V_{\xi_0}(K)-1\geq 1.\]
In Rasmussen's notation, we conclude that $\mathfrak{A}_{\xi_0+PD[\mu]}$ is of type $\ast$. Since Proposition \ref{Prop:L-spacesurgerycriteria} implies that there is at most one $\ast$ for $L$-space surgeries, we conclude 
$$N(K, \gamma, \xi+PD[\mu])=\min  \{V_{\xi_0+PD[\mu]}(K),H_{\xi_0+PD[\mu]}(K) \}$$ 
by Proposition \ref{Prop:computeN}.  If we let $N_1 = N(K, \gamma, \xi+PD[\mu])$, then clearly $N_0-1\leq N_1 \leq N_0$.  Finally, (\ref{dinv2}) follows from applying Corollary \ref{Cor:d-inv surgery formula1} to $\xi_0+PD[\mu]$.

The other case, where $V_{\xi_0}(K) \geq H_{\xi_0}(K) \geq 2$, can be treated symmetrically by considering the relative ${\rm Spin^c}$ structure $\xi_0-PD[\mu]$ instead.  This concludes the proof.  
\end{comment}
   
\end{proof}

\section{The Casson-Walker invariant}

To investigate our surgery problem, we begin by introducing the Casson-Walker invariant.  This invariant is closely related to the $d$-invariant for $L$-spaces, as demonstrated in the following theorem. It is a special case of results found in \cite[Theorem 3.3]{Rus} and \cite[Theorem 5.1]{OS1}. 

We adopt the convention that $\lambda(P)=1$, where $P$ is the positively oriented Poincar\'e homology sphere, obtained by performing $+1$-surgery along the right-handed trefoil knot. Note that the Casson-Walker invariants change sign under orientation reversal.

\begin{theorem}
\label{Cad}
For an L-space $Y$, the following holds:
\[-2|H_1(Y;\mathbb{Z})| \lambda(Y)=\sum \limits_{\mathfrak{s} \in {\rm Spin^c}(Y)} d(Y,\mathfrak{s}).\]
\end{theorem}

In general, the Casson-Walker invariant is easier to compute than the $d$-invariant, though less powerful in applications. To simplify calculations in certain cases in the subsequent sections, we will apply our $d$-invariant surgery formula as a sum over all $\rm Spin^c$ structures, utilizing the Casson-Walker invariant as stated in Theorem \ref{Cad}. %rather than restricting ourselves to self-conjugate $\rm Spin^c$ structures.

\subsection{The Casson-Walker invariant for lens spaces} The Casson-Walker invariant of a lens space can be computed using the Dedekind sum.

\begin{definition}
For coprime integers $p$ and $q$, the \textit{Dedekind sum} $s(q,p)$ is defined as
\begin{equation*}
s(q,p)={\rm sign}(p) \cdot \sum \limits^{|p|-1}_{k=1}\left(\!\left(\frac{k}{p}\right)\!\right) 
\left(\!\left(\frac{kq}{p}\right)\!\right), 
\end{equation*}
where 
\begin{equation*}
\left(\!\left(x\right)\!\right)=\left\{
\begin{array}{ll}
x- \lfloor x \rfloor - \frac{1}{2}  & {\rm if } \,\, x \notin \mathbb{Z},\\
 0  & {\rm if } \,\, x \in \mathbb{Z}. \\
\end{array} \right.
\end{equation*}
\end{definition}

\begin{prop}\cite[Theorem 2.8]{BL}
\label{CaLe}
For a lens space $L(p,q)$, we have 
$\lambda(L(p,q))=-\frac{1}{2}s(q,p)$.
\end{prop}

The Dedekind sum $s(q,p)$ can be calculated as follows (cf. \cite[Lemma 4.3]{Ras3}). For $p>q>0$, express $p/q$ using the Hirzebruch-Jung continued fraction expansion

\begin{equation*}
p/q=[a_1,a_2,\dots, a_n]=a_1-\dfrac{1}{a_2-\dfrac{1}{\dots - \dfrac{1}{a_n}}}, 
\end{equation*}
where $a_i\geq 2$. Then
\begin{equation}
\label{spq}
s(q,p)=\frac{1}{12}\left( \frac{q}{p}+\frac{q'}{p}+\sum \limits^{n}_{i=1}(a_i-3) \right),
\end{equation}
where $0<q'<p$ is the unique integer satisfying $qq' \equiv 1 \pmod p$.

\subsection{The Casson-Walker invariant for Seifert fibered spaces}
\label{The Casson-Walker invariant for Seifert fibered spaces}

Lescop \cite[Proposition 6.1.1]{Les} provides the following formula for the Casson-Walker invariant of a Seifert fibered space 
$$M=M(0,0;b_1/a_1, \dots, b_n/a_n)$$
over $S^2$ with $n$ exceptional fibers, where $a_i>1$ and $\gcd(a_i,b_i)=1$:  %, which is presented by the surgery diagram shown in \textbf{Figure \ref{}}. 
\begin{equation}\label{CWiS}
\lambda(M)=\frac{1}{|H_1(M)|} \left( \frac{{\rm sign}(e)}{24} \left(2-n+\sum \limits^{n}_{i=1} \frac{1}{a^2_i}\right)+\frac{e|e|}{24}-\frac{e}{8}-\frac{|e|}{2} \sum \limits^{n}_{i=1} s(b_i,a_i) \right) |\prod \limits^{n}_{i=1} a_i|,
\end{equation}
where $e=\sum \limits^{n}_{i=1} b_i/a_i$ is the Euler number of $M$, and $s(b_i,a_i)$ are the Dedekind sums.

Let us apply this formula to compute the Casson-Walker invariant of the Seifert fibered space 
\[M=M\left(0,0; \frac{1}{n-k}, \frac{1}{k}, \frac{1}{m-k}\right)\]
with $k>1$ and $m\leq k-3$.  This computation will be used in Section \ref{mleqk-3s} to address the distance one surgery problem (\textbf{Case F}). 

We have the Euler number $e$ given by:
\begin{equation*}
e=\frac{1}{n-k}+\frac{1}{k}-\frac{1}{k-m}=\frac{k^2-mn}{(k-m)(n-k)k}. 
\end{equation*}

For the case $e>0$, which is relevant for our study in Section \ref{mleqk-3s}, we have
\begin{align*}
\sum \limits^{3}_{i=1} s(b_i,a_i)&=s(1,k)+s(1,n-k)+s(-1,k-m)\\
&=\left(\frac{1}{6k}+\frac{k}{12}-\frac{1}{4}\right)+\left(\frac{1}{6(n-k)}+\frac{n-k}{12}-\frac{1}{4}\right)-\left(\frac{1}{6(k-m)}+\frac{k-m}{12}-\frac{1}{4}\right)\\
%&=\frac{1}{6k}+\frac{1}{6(n-k)}-\frac{1}{6(k-m)}+\frac{n+m-k}{12}-\frac{1}{4}\\
%&=\frac{e}{6}+\frac{1}{6}+\frac{n+m-k}{12}-\frac{5}{12}\\
&=\frac{e}{6}+\frac{m+n-k}{12}-\frac{1}{4}.
\end{align*}
From (\ref{eqmn-k2}), we have $|H_1(M)|=|mn-k^2|=k^2-mn$, given $e>0$. Finally, we use Formula (\ref{CWiS}) to compute the Casson-Walker invariant:  
\begin{equation}\label{lambdaM}
\lambda(M)=-\frac{(k-m)(n-k)k}{24(k^2-mn)}+\frac{m+n-k}{12(k^2-mn)}-\frac{m+n-k}{24}.
\end{equation}

\begin{proof}
This follows from straightforward but tedious computation:
\begin{equation*}
\begin{aligned}
&\lambda(M)=\frac{|\prod \limits^{3}_{i=1} a_i|}{|H_1(M)|} \left( \frac{{\rm sign}(e)}{24} \left(2-3+\sum \limits^{3}_{i=1} \frac{1}{a^2_i}\right)+\frac{e|e|}{24}-\frac{e}{8}-\frac{|e|}{2} \sum \limits^{3}_{i=1} s(b_i,a_i) \right) \\
&=\frac{(k-m)(n-k)k}{k^2-mn} \left(\frac{1}{24}\left(-1+\frac{1}{(n-k)^2}+\frac{1}{k^2}+\frac{1}{(k-m)^2} \right)+\frac{e^2}{24}-\frac{e}{8}-\frac{e}{2}\left(\frac{e}{6}+\frac{m+n-k}{12}-\frac{1}{4}\right)\right)\\
&=\frac{1}{e}\left(\frac{1}{24}\left(-1+\frac{1}{(n-k)^2}+\frac{1}{k^2}+\frac{1}{(k-m)^2} \right)-\frac{e^2}{24}-\frac{e(m+n-k)}{24}\right)\\
&=-\frac{1}{24e}+\frac{1}{24e}\left(\frac{1}{(n-k)^2}+\frac{1}{k^2}+\frac{1}{(k-m)^2} \right)-\frac{e}{24}-\frac{m+n-k}{24}\\
&=-\frac{1}{24e}+\frac{1}{24e}\left( 
 \left(\frac{1}{n-k}+\frac{1}{k}-\frac{1}{k-m}\right)^2-2\left(\frac{1}{k(n-k)}-\frac{1}{k(k-m)}-\frac{1}{(n-k)(k-m)}\right) \right)\\
& \quad -\frac{e}{24}-\frac{m+n-k}{24}\\
&=-\frac{1}{24e}+\frac{e}{24}-\frac{1}{12}\cdot\frac{k(n-k)(k-m)}{k^2-mn}\cdot\frac{k-m-n}{k(n-k)(k-m)}-\frac{e}{24}-\frac{m+n-k}{24}\\
&=-\frac{(k-m)(n-k)k}{24(k^2-mn)}+\frac{m+n-k}{12(k^2-mn)}-\frac{m+n-k}{24},
\end{aligned}
\end{equation*}
where we used $e=\frac{1}{n-k}+\frac{1}{k}-\frac{1}{k-m}=\frac{k^2-mn}{(k-m)(n-k)k}$.

\end{proof}

\subsection{Casson-Walker invariant as surgery obstruction}

We establish the following constraint on surgeries using the Casson-Walker invariant.  

\begin{proposition}
\label{CassonWalkerObstruction}
Suppose $K'\preceq K$ and $\gamma$ is a positive framing such that $Y_\gamma(K)$ is an $L$-space.  Then the Casson-Walker invariants satisfy
\begin{equation}\label{inequality:CassonWalker}
\Delta\lambda:=\lambda(Y_\gamma(K))-\lambda(Y_\gamma(K')) \geq 0.
\end{equation}
Moreover, $|H_1(Y_\gamma(K))|\cdot \Delta \lambda$ is a non-negative integer.  
    
\end{proposition}

\begin{proof}

Summing the $d$-invariant inequality (\ref{Ineq:d-invpartialorder}) over all $\rm Spin^c$ structures, we have
$$\sum \limits_{\mathfrak{s} \in {\rm Spin^c}(Y_\gamma(K)} d(Y_\gamma(K),\mathfrak{s})\leq \sum \limits_{\mathfrak{s} \in {\rm Spin^c}(Y_\gamma(K'))} d(Y_\gamma(K'), \mathfrak{s}).$$
Note that $Y_\gamma(K')$ is also an $L$-space by Theorem \ref{Thm:L-spaceknotpartialorder}.  Applying Theorem \ref{Cad}, we can rewrite this inequality in terms of the Casson-Walker invariants, yielding
    $\lambda(Y_\gamma(K)) \geq \lambda(Y_\gamma(K'))$.  The integrality of $|H_1(Y_\gamma(K))|\cdot \Delta \lambda$ follows from (\ref{eq:L-spacesurgeryd-inv}).    
    \end{proof}

\section{Organization and main tools for proof of Theorem \ref{L(n,1)maintheorem}}
\label{Proof of main theorem}

We use $d$-invariant formulas to study the distance one surgery problem. The investigation naturally divides into several cases, each requiring distinct computations of $d$-invariants and specific versions of $d$-invariant formula. Moreover, the type of the resulting space $M$ in the surgery formula from Proposition \ref{prop rational d-inv1} - whether a lens space or a Seifert fibered space over $S^2$ with $3$ exceptional fibers - depends on the specific values of $k$ and $m$.  The cases are divided based on: 

\begin{enumerate}
    \item[1.] The parity relationship between $n$ and $s$.
    \item[2.] Whether $K$ is null-homologous or homologically essential.
    \item[3.] The values of $k$ and $m$ that determine the surgery outcome space $M$.
\end{enumerate}

\noindent
By symmetry, a distance one surgery from $L(n,1)$ to $L(s,1)$ implies: 
\begin{itemize}
   
\item The existence of a distance one surgery from $L(s,1)$ to $L(n,1)$ via the dual surgery.

\item The existence of a distance one surgery from $L(-n.1)$ to $L(-s, 1)$ via mirror symmetry. 

\end{itemize}
Therefore, it suffices to consider the case where $n\geq |s|>0$, as assumed in Theorem \ref{L(n,1)maintheorem}. In the subsequent sections, we will conduct a case-by-case study. For the readers' convenience, we summarize the results and main tools for each case in Table \ref{cases}.

\begin{table}[H] %H !htbp
\centering
\begin{tabular}{c|c|c|c|c|c}
\multicolumn{4}{c|}{Cases}& Main Tools & Proved in \\
\hline
\multicolumn{4}{c|}{%Case\uppercase\expandafter{\romannumeral1}:
$n$ and $s$ have different parities \textbf{(A)}} & \tabincell{c}{Lin's formula \\ (Lemma \ref{lemmaLMV})} &\tabincell{c}{Proposition \ref{L(n,1)propdiffa} \\ (in Section \ref{Distance one surgeries between $L(n,1)$ and $L(s,1)$ with $n$ and $s$ different parities})} \\
\hline
\multirow{5}*{\tabincell{c}{
%Case\uppercase\expandafter{\romannumeral2}: 
$n$ and \\ $s$ have the \\ same parity}}& \multicolumn{3}{c|}{%\uppercase\expandafter{\romannumeral2}.\romannumeral1: 
Null-homologous \textbf{(B)}} & \tabincell{c}{ Ni-Wu's formula \\ (Proposition \ref{NiWu})} 
& \tabincell{c}{Proposition \ref{nullha} \\ (in Section \ref{$K$ is null-homologous})} \\
\cline{2-6}
&\multirow{4}*{\tabincell{c}{ %\uppercase\expandafter{\romannumeral2}.\romannumeral2:\\ 
Homologically \\ essential}} &\multicolumn{2}{c|}{%\uppercase\expandafter{\romannumeral2}.\romannumeral2.(\romannumeral1): 
$k=1$ \textbf{(C)}} & Assumption $|s| \leq n$
& \tabincell{c}{Proposition \ref{k=1s}  \\ (in Section \ref{hessentialk=1})}\\
\cline{3-6}
& &\multirow{3}*{\tabincell{c}{%\uppercase\expandafter{\romannumeral2}.\romannumeral2.(\romannumeral2):\\ 
$k>1$}} & \tabincell{c}{%\vspace{1mm} %\uppercase\expandafter{\romannumeral2}.\romannumeral2.(\romannumeral2).(a):\\ 
$m \geq k+1$ \textbf{(D)}} & 
Assumption $|s| \leq n$
& \tabincell{c}{No valid surgery \\ (in Section \ref{hessentialk=1})} \\
\cline{4-6}
& & & \tabincell{c}{%\vspace{0.5mm} %\uppercase\expandafter{\romannumeral2}.\romannumeral2.(\romannumeral2).(b):\\ 
$m=k-1$ \textbf{(E)}} 
& \tabincell{c}{$L$-space surgery formula \\ (Proposition \ref{prop rational d-inv1}),\\
Casson-Walker invariant \\ (Proposition \ref{CassonWalkerObstruction}) }
& \tabincell{c}{Proposition \ref{m=k-1n} \\ (in Section \ref{m=k-1s})}\\
\cline{4-6}
& & & \tabincell{c}{%\vspace{1mm} %\uppercase\expandafter{\romannumeral2}.\romannumeral2.(\romannumeral2).(c):\\ 
$m \leq k-3$ \textbf{(F)}} & \tabincell{c}{Classification of Seifert \\ fibered $L$-space, \\ Casson-Walker invariant\\(Proposition \ref{CassonWalkerObstruction})}
& \tabincell{c}{Proposition \ref{mleqk-3n} \\ (in Section \ref{mleqk-3s})} \\
\cline{2-6}
% &\multicolumn{3}{c|}{The rest} & Considering symmetry of surgeries & Section \ref{Proof of main theorem}\\
\hline
%a &\multicolumn{2}{|c}{Proposition \ref{L(n,1)propeven}} \\
%\hline
%\multirow{2}*{Odd Case}& null-homologous knot & b\\  
%\cline{2-3} 
%& homologically essential knot & v\\
\end{tabular}
\vspace{1mm}
\caption{A summary of methods and results for each case.}
\label{cases}
\end{table}

\section{Surgeries between $L(n,1)$ and $L(s,1)$ with different parities.}
\label{Distance one surgeries between $L(n,1)$ and $L(s,1)$ with $n$ and $s$ different parities}

This section examines distance one surgeries between lens spaces $L(n,1)$ and $L(s,1)$ where $n$ and $s$ have different parities, following the approach in \cite[Section 2]{WY}. Our argument builds on two key lemmas:

%Due to the symmetry of the distance one surgeries, it suffices to consider surgeries between $L(n,1)$ and $L(s,1)$, where $n>0$ is odd and $s \neq 0$ is even. 

%For $n=1$, there always exists a distance one surgery from $L(s,1)$ to $L(1,1) \cong S^3$ for any even $s \neq 0$, which can be realized as the double branched cover of the band surgery shown in Figure \ref{bandsurgery4}. The case $n=3$ has been studied in \cite{LMV}. Hence, we restrict our attention to the case where $n \geq 5$ is odd. 

%Using an argument nearly identical to that in \cite[Lemma 2.1]{WY}, we establish the following lemma.

\begin{lemma} \cite[Lemma 2.1]{WY}
\label{lemma1}
Let $Y'$ be obtained by a distance one surgery from $Y=L(n,1)$ with $n \geq 5$ odd, and let $W: Y \rightarrow Y'$ be the associated cobordism. Then $|H_1(Y')|$ is even if and only if $W$ is $\rm Spin$.
\end{lemma}

\begin{lemma} \cite[Theorem 7]{Lin} \cite[Lemma 2.7]{LMV}
\label{lemmaLMV}
For a {\rm Spin} cobordism $(W,\mathfrak{s}):(Y,\mathfrak{t}) \rightarrow (Y',\mathfrak{t}')$ between L-spaces where $b_2^+(W)=1$ and $b_2^-(W)=0$,
\[d(Y',\mathfrak{t}')-d(Y,\mathfrak{t})=-\dfrac{1}{4}.\]
\end{lemma}

\noindent
Using these lemmas, we obtain: 
\begin{prop}
\label{L(n,1)propdiff}
For odd $n\geq 5$ and even $s\neq 0$, the lens space $L(s,1)$ can be obtained from $L(n,1)$ by a distance one surgery along a knot if and only if $s=n\pm 1$.
\end{prop}

\begin{proof}[Sketch of proof:]
Consider the surgery cobordism $W: L(n,1) \rightarrow L(s,1)$.  By Lemma \ref{lemma1}, $W$ is Spin. Applying Lemma \ref{lemmaLMV}: 
$$d(L(s,1),i)-d(L(n,1),0)=\pm \frac{1}{4},$$
where $i=0$ or $|s|/2$, and the sign depends on whether $b_2^-(W)=1$ or $b_2^+(W)=1$. Using the $d$-invariant formula (\ref{eq2}), we obtain $s=n\pm 1$. (See \cite[pp 273-274]{WY} for complete details.)
\end{proof}

Combining this with Lidman, Moore, and Vazquez's result on $L(3,1)$ \cite[Theorem 1.1]{LMV}, we can fully characterize distance one surgeries between $L(n,1)$ and $L(s,1)$ with different parities. For simplicity, we present only cases where $n \geq |s|>0$, as other cases follow by surgery symmetry.

\begin{prop}
\label{L(n,1)propdiffa}
The lens space $L(s,1)$ can be obtained by a distance one surgery along a knot in $L(n,1)$, where $n \geq |s|>0$ have different parities, if and only if $(n,s)$ satisfies one of:
\begin{itemize}
\item[(1)] $n$ is any positive even integer and $s=\pm 1$;
\item[(2)] $n$ is any positive integer and $s=n-1$;
\item[(3)] $n=3$ and $s=-2$;
\item[(4)] $n=6$ and $s=-3$.
\end{itemize}
\end{prop}

\noindent
Note that all these surgeries can be realized as double branch covers of the band surgeries shown in Figure
\ref{bandsurgery}.  Thus, we resolved Case \textbf{(A)} in Table \ref{cases}.

\section{Distance one surgeries along null-homologous knots}
\label{$K$ is null-homologous}

In this section, we apply Ni-Wu's $d$-invariant surgery formula \cite[Proposition 1.6]{NiWu} to analyze distance one surgeries along null-homologous knots. 

%Let $\mathfrak{t}_i\in {\rm Spin^c}(Y_m(K))$ denote the ${\rm Spin^c}$ structure that corresponds to $(\mathfrak{t},i) \in {\rm Spin^c}(Y)\oplus \mathbb{Z}_m$ under the natural bijection between the sets of $\rm Spin^c$ structures ${\rm Spin^c}(Y_m(K))$ and ${\rm Spin^c}(Y) \oplus \mathbb{Z}_m$.

Let $W_m(K): Y \rightarrow Y_m(K)$ be the cobordism obtained by attaching a 2-handle to $[0,1] \times Y$ with framing $m$ along a null-homologous knot $K$. There is a natural bijection between the sets of $\rm Spin^c$ structures on $Y_m(K)$, denoted by ${\rm Spin^c}(Y_m(K))$, and ${\rm Spin^c}(Y) \oplus \mathbb{Z}_m$ as follows: 
each $(\mathfrak{t},i) \in {\rm Spin^c}(Y)\oplus \mathbb{Z}_m$ corresponds to the unique $\rm Spin^c$ structure on $Y_m(K)$ that  extends to a $\rm Spin^c$ structure $\mathfrak{v}$ over $W_m(K)$ satisfying:
\begin{equation}
\label{eqns}
   \langle c_1(\mathfrak{v}), [\widehat{F}] \rangle +m \equiv 2i \Mod{2m}
\end{equation}
where $[\widehat{F}]$ is the surface obtained by capping off a Seifert surface for $K$, and $\mathfrak{v}|_Y=\mathfrak{t}$. We denote this $\rm Spin^c$ structure by $\mathfrak{t}_i$.

\begin{lemma}
\label{nhss}
For a null-homologous knot $K$ in $Y$, if $\mathfrak{t}$ is a self-conjugate $\rm Spin^c$  structure on $Y$, then $\mathfrak{t}_0 \in {\rm Spin^c}(Y_m(K))$, identified with $(\mathfrak{t},0) \in {\rm Spin^c}(Y)\oplus \mathbb{Z}_m$, is also self-conjugate.
\end{lemma}

\begin{proof}
Let $\mathfrak{v}$ be a $\rm Spin^c$ structure on $W_m(K)$ that extends $\mathfrak{t}_0 \in  {\rm Spin^c}(Y_m(K))$ over $W_m(K)$. Then
\[\langle c_1(\mathfrak{v}), [\widehat{F}] \rangle +m \equiv 0 \Mod{2m}.\]
The conjugate $\rm Spin^c$ structure $\bar{\mathfrak{v}}$ extends $\bar{\mathfrak{t}}_0 \in  {\rm Spin^c}(Y_m(K))$ and $\bar{\mathfrak{v}}|_Y=\bar{\mathfrak{t}}=\mathfrak{t}$, since we assume $\mathfrak{t}$ is self-conjugate. Therefore, 
\[\langle c_1(\bar{\mathfrak{v}}), [\widehat{F}] \rangle=-\langle c_1(\mathfrak{v}), [\widehat{F}] \rangle \equiv m  \Mod{2m},\]
which implies
\[\langle c_1(\bar{\mathfrak{v}}), [\widehat{F}] \rangle +m \equiv 2m \equiv 0 \Mod{2m}.\]
Thus, both $\bar{\mathfrak{t}}_0$ and $\mathfrak{t}_0$ extend over $W_m(K)$ agreeing with $\mathfrak{t}$ on $Y$ and correspond to $i=0$. Hence $\bar{\mathfrak{t}}_0=\mathfrak{t}_0$, proving that $\mathfrak{t}_0$ is self-conjugate.
\end{proof}

\begin{prop}\cite[Proposition 1.6]{NiWu}
\label{NiWu}
For a null-homologous knot $K$ in an L-space $Y$, fix $m>0$ and a $\rm Spin^c$ structure $\mathfrak{t}$ on $Y$. Then for any $\mathfrak{t}_i \in {\rm Spin^c}(Y_m(K))$ with $i \in \mathbb{Z}_m$, we have
%Then there exists a bijective correspondence $i \leftrightarrow \mathfrak{t}_i$ between $\mathbb{Z}/m\mathbb{Z}$ and the $\rm Spin^c$ structures on ${\rm Spin^c}(Y_m(K))$ that extend $\mathfrak{t}$ over the cobordism $W_m(K): Y \rightarrow$ $Y_m(K)$ such that
\begin{equation}
d(Y_m(K), \mathfrak{t}_i)=d(Y,\mathfrak{t})+d(L(m,1),i)-2\max\{V_{\mathfrak{t},i}(K),V_{\mathfrak{t},m-i}(K)\}. \label{eq4}
\end{equation}
%where $N_{\mathfrak{t},i}=\max\{V_{\mathfrak{t},i}(K),V_{\mathfrak{t},m-i}(K)\}$.
\end{prop}

Here, $V_{\mathfrak{t},i}(K)$ corresponds to $V_{\xi}(K)$ in Section \ref{mapping cone for rationally null-homologous knots}, with $(\mathfrak{t}, i)$ denoting the relative $\rm Spin^c$ structure $\xi$ with $G_{Y,K}(\xi)=\mathfrak{t}$ and Alexander grading $i$.  It satisfies the monotonicity property
 
\begin{equation}
\label{Vm}
V_{\mathfrak{t},i}(K) \geq V_{\mathfrak{t},i+1}(K) \geq V_{\mathfrak{t},i}(K)-1.
\end{equation}
%\end{property}

\medskip

Now we examine distance one surgeries between lens spaces $L(n,1)$ and $L(s,1)$, where $n\geq |s|>0$ have the same parity. Since $|s|=|H_1(Y_m(K))|=mn$, we must have $|m|=1$ and $|s|=n$. When $s=n$, distance one surgeries from $L(n,1)$ to itself certainly exist. We therefore focus on the case $s=-n$.

\begin{prop}
\label{L(n,1)nullhomologous}
For odd $n>0$, the lens space $L(-n,1)$ can be obtained by a distance one surgery along a null-homologous knot $K$ in $L(n,1)$ if and only if $n=1$ or $5$.
\end{prop}

\begin{proof}
Moore and Vazquez \cite[Corollary 3.7]{MV} proved that for \textit{square-free} odd $n>0$, distance one surgery from $L(n,1)$ to $L(-n,1)$ exists if and only if $n=1$ or $5$. In their proof, the square-free condition was used only to show that any knot admitting such surgery must be null-homologous. Since we assume $K$ is null-homologous, their result applies directly, ruling out $s=-n$ except when $n=1$ and $5$. 

%The surgery between $L(1,1) \cong L(-1,1) \cong S^3$ is given by the double branched cover of the band surgery in Figure \ref{bandsurgery1}.
%The $L(5,1)$ to $L(-5,1)$ surgery is realized by the double branched cover of Zekovi\'{c}'s band surgery, as illustrated in Figure \ref{bandsurgeryT25}.

\end{proof}

\begin{prop}
\label{n1-n}
For even $n>0$, the lens space $L(-n,1)$ can be obtained by a distance one surgery along a null-homologous knot $K$ in $L(n,1)$ only if $n=2$ or $10$.
\end{prop}

\begin{proof}
From our earlier discussion on first homology, the surgery slope $m$ must be $\pm 1$.  

\medskip
\noindent
\textit{Case 1}: $+1$-surgery from $L(n,1)$ to $L(-n,1)$

Let $\mathfrak{t}$ and $ \mathfrak{s}$ denote the self-conjugate $\rm Spin^c$ structures identified with $0$ and $\frac{n}{2} \in{\rm Spin^c}(L(n,1)) \cong \mathbb{Z}_n$ respectively. By Lemma \ref{nhss}, both $\mathfrak{t}_0$ and $\mathfrak{s}_0 \in{\rm Spin^c}(L(-n,1))$ are self-conjugate, so $$\{ \mathfrak{t}_0, \mathfrak{s}_0 \} =\{ 0, \frac{n}{2}\}\in {\rm Spin^c}(L(-n,1)) \cong \mathbb{Z}_n,$$
yielding two possibilities:
\begin{comment}
\begin{itemize}
\item[a.] The $\rm Spin^c$ structure $\mathfrak{t}_0$ corresponds to $0$ and $\mathfrak{s}_0$ corresponds to $\frac{n}{2}$;
\item[b.] The $\rm Spin^c$ structure $\mathfrak{t}_0$ corresponds to $\frac{n}{2}$ and $\mathfrak{s}_0$ corresponds to $0$.
\end{itemize}
\end{comment}

\medskip
\noindent
\textit{Subcase 1a}: $\mathfrak{t}_0=0$ and $\mathfrak{s}_0=\frac{n}{2}$. %Applying (\ref{eq4}) for $\mathfrak{s}=\frac{n}{2} \in {\rm Spin^c}(L(n,1))$ and $\mathfrak{s}_0=\frac{n}{2} \in {\rm Spin^c}(L(-n,1))$, 

Applying (\ref{eq4}) to $\mathfrak{s}=\frac{n}{2} \in {\rm Spin^c}(L(n,1))\cong \mathbb{Z}_n$, we have
\begin{equation*}
d(L(-n,1),\frac{n}{2})=d(L(n,1),\frac{n}{2})+d(L(1,1),0)-2\max \{V_{\mathfrak{s},0}(K),V_{\mathfrak{s},1}(K)\}.
\end{equation*}
Using (\ref{eq2}), we find 
$$\max \{V_{\mathfrak{s},0}(K),V_{\mathfrak{s},1}(K)\}=-\frac{1}{4} <0.$$ 
This contradicts the non-negativity of $V$-invariants. 

\medskip
\noindent
\textit{Subcase 1b}: $\mathfrak{t}_0=\frac{n}{2}$ and $\mathfrak{s}_0=0$. 

Applying (\ref{eq4}) to $\mathfrak{t}=0 \in {\rm Spin^c}(L(n,1))$, we have
\begin{equation*}
d(L(-n,1),\frac{n}{2})=d(L(n,1),0)+d(L(1,1),0)-2\max \{V_{\mathfrak{t},0}(K),V_{\mathfrak{t},1}(K)\}.
\end{equation*}
Using (\ref{eq2}), we find 
$$V_{\mathfrak{t},0}(K)=\max \{V_{\mathfrak{t},0}(K),V_{\mathfrak{t},1}(K)\}=\frac{n-2}{8}.$$ 
Since L-space surgery slope $m$ satisfies $m \geq 2g(K)-1$, and $m=1$ here, we deduce that $g(K)=0$ or $1$. Let 
\[A_{\max}(Y,K)=\max\{A_{Y,K}(\xi)|\xi \in \underline{\rm Spin^c}(Y,K) \,\,\, {\rm with} \,\,\, \widehat{HFK}(Y,K,\xi) \neq 0 \}.
%\,\,\, {\rm and} \,\,\, 
\]
By \cite[Theorem 1.1]{Ni}, $A_{\max}(Y,K)=g(K)=0$ or $1$. Since $V_{\mathfrak{t},A_{\max}(Y,K)}(K)=0$, property (\ref{Vm}) implies $V_{\mathfrak{t},0}(K)=0$ or $1$, yielding $n=2$ or $n=10$.

\medskip \noindent
%Next, consider the case that $L(-n,1)$ is obtained by a $-1$-surgery along a null-homologous knot $K$ in $L(n,1)$. 
\textit{Case 2}: $-1$-surgery from $L(n,1)$ to $L(-n,1)$

If we reverse the orientation, then $L(n,1)$ is obtained by $+1$-surgery along the mirror knot $\overline{K}$ in $L(-n,1)$. Let $\mathfrak{t}$ and $ \mathfrak{s}$ denote the self-conjugate $\rm Spin^c$ structures corresponding to $0, \frac{n}{2} \in{\rm Spin^c}(L(-n,1)) \cong \mathbb{Z}_n$ respectively. As before, there are two possibilities:

\medskip\noindent
\textit{Subcase 2a}: $\mathfrak{t}_0=0$ and $\mathfrak{s}_0=\frac{n}{2}$. 

Applying (\ref{eq4}) to $\mathfrak{t}=0 \in {\rm Spin^c}(L(-n,1))$, we obtain
\begin{equation*}
d(L(n,1),0)=d(L(-n,1),0)+d(L(1,1),0)-2\max \{V_{\mathfrak{t},0}(\overline{K}),V_{\mathfrak{t},1}(\overline{K})\}.
\end{equation*}
Calculation shows 
$$\max \{V_{\mathfrak{t},0}(\overline{K}),V_{\mathfrak{t},1}(\overline{K})\}= \frac{1-n}{4}<0,$$ 
which contradicts the non-negativity of $V$-invariants. 

\medskip \noindent
\textit{Subcase 2b}: $\mathfrak{t}_0=\frac{n}{2}$ and $\mathfrak{s}_0=0$. 

Applying (\ref{eq4}) to $\mathfrak{t}=0 \in {\rm Spin^c}(L(-n,1))$, we have
\begin{equation*}
d(L(n,1),\frac{n}{2})=d(L(-n,1),0)+d(L(1,1),0)-2\max \{V_{\mathfrak{t},0}(\overline{K}),V_{\mathfrak{t},1}(\overline{K})\}.
\end{equation*}
A similar calculation yields
$$\max \{V_{\mathfrak{t},0}(\overline{K}),V_{\mathfrak{t},1}(\overline{K})\}=\frac{2-n}{8}<0$$ 
except when $n=2$. This completes the proof.
\end{proof}

%\begin{remark}
%Since $L(2,1) \cong L(-2,1)$, a distance one surgery between them exists, which can be realized as the double branched cover of the band surgery shown in Figure \ref{bandsurgery1}. 
%\end{remark}

In summary, we have the following result for distance one surgeries along null-homologous knots, which settles Case \textbf{(B)} in Table \ref{cases}.

\begin{prop}
\label{nullha}
For distance one surgeries along null-homologous knots in $L(n,1)$ yielding $L(s,1)$, where $n \geq |s|>0$ have the same parity, the possible pairs $(n,s)$ are:
\begin{itemize}
\item[(1)] $s=n$;
\item[(2)] $n=1$ and $s=-1$;
\item[(3)] $n=2$ and $s=-2$;
\item[(4)] $n=5$ and $s=-5$;
\item[(5)] $n=10$ and $s=-10$.
\end{itemize}
\end{prop}

\section{Distance one surgeries along homologically essential knots}
\label{$K$ is homologically essential}
In this section, we study distance one surgeries between $L(n,1)$ and $L(s,1)$ along homologically essential knots, where $n \geq |s|>0$ have the same parity.  
%There are two cases to consider:
%\begin{itemize}
%\item[1.] Both $n$ and $s$ are odd;
%\item[2.] Both $n$ and $s$ are even.
%\end{itemize}
According to equation (\ref{eqmn-k2}),
\begin{equation*} 
|s|=|H_1(Y_{m\mu+ \lambda}(K))|%=\left|det 
 %\begin{bmatrix}
 %\dfrac{mn-k^2}{d} & kk'-n'm \\
 %0 & d \\
%\end{bmatrix}\right|
=|mn-k^2|.
\end{equation*}
Furthermore, (\ref{gcd}) implies  $$\gcd\left(\frac{mn-k^2}{d},\, kk'-n'm,\, d \right)=1,$$
where $d=\gcd(n,k)$.  We conclude that:

\begin{lemma}\label{lemmaparity}
Suppose $m$-surgery along a homologically essential knot $K$ in $L(n,1)$ with winding number $k$ yields $L(s,1)$, and $n$ and $s$ have the same parity.  Then: \begin{enumerate}
\item When $n$ and $s$ are odd, $k$ and $m$ must have different parities;

\item When $n$ and $s$ are even, $k$ must be even and $m$ must be odd.
\end{enumerate}
\end{lemma}

\subsection{Case (C) $k=1$} 
\label{hessentialk=1}
From Lemma \ref{lemmaparity}, we see that $n$ and $s$ must be odd, and $m$ must be even.  

If $m=0$, then $|s|=|mn-k^2|=1$; 

If $m \geq 2$, then $|s|=|mn-k^2|=mn-1 \geq 2n-1 >n$ for $n>1$; 

If $m \leq -2$, then $|s|=|mn-k^2| \geq 2n+1 >n$.
 %which contradicts our assumption. 
 
 \medskip
Consequently, we have:

\begin{prop}
\label{k=1s}
Suppose $n \geq |s|>0$ of the same parity. The lens space $L(s,1)$ can be obtained by a distance one surgery along a knot of winding number $k=1$ in $L(n,1)$ if and only if $s=\pm 1$ and $n$ is any positive odd integer.
\end{prop}

%\subsection
\medskip 
\noindent
\textbf{Case (D) $k>1$ and $m\geq k+1$.} 
%\label{mgeqk+1}
We have $$|s|=|mn-k^2| \geq (k+1)n-k^2=n+k(n-k)>n.$$  Therefore, this case contradicts our assumption $|s|\leq n$, and no valid surgeries exist.

% \subsubsection{Case \uppercase\expandafter{\romannumeral2}.\romannumeral2.(b) $k>1$ and $m=k+1$} In this case, $|s|=|mn-k^2|=n+nk-k^2>n$ since $n>k$. Therefore no desired surgery in this case.

\subsection{Case (E) $k>1$ and $m=k-1$} \label{m=k-1s}
In this case, we have $$|s| = |(k-1)n-k^2|=k(n-k)-n.$$ 

%when $n>4$, and simple algebra yields  

\begin{lemma}\label{CaseFnkvalue}
For Case (E), the conditions $0<|s| \leq n$, combined with $k>1$ and our standing assumption $k\leq \lfloor n/2 \rfloor$ on the winding number, restrict the values of $n$ and $k$ to the following possibilities:
\begin{itemize}
\item[(1)] $k=2$ and $n\geq5$;
\item[(2)] $k=3$ and $n\in \{6,7,8,9\}$;
\item[(3)] $k=4$ and $n=8$.
\end{itemize}

\end{lemma}

\medskip\noindent
Note that we can have either $s>0$ or $s<0$.  For positive $s$, we have:

\begin{prop}
\label{propm=k-1s>0}
For $n \geq s>0$ of the same parity, the lens space $L(s,1)$ can be obtained by a distance one surgery along a knot in $L(n,1)$ with winding number $k>1$ and surgery coefficient $m=k-1$ if and only if $n \geq 5$ and $s=n-4$.
\end{prop}
\begin{proof}
%We apply Casson-Walker invariant to analyze this case. 
For $k=2$, we have $s=|mn-k^2|=n-4$. The distance one surgery from $L(n,1)$ to $L(n-4,1)$ exists and can be realized as the double branched cover of the band surgery shown in Figure \ref{bandsurgery5}.  %Alternatively, $1$-surgery along the simple knot $K(n,1,2) \subset L(n,1)$ also yields $L(n-4,1)$. 

We now prove that there is no other valid surgery for $k=3$ or $4$. Note that for $m=k-1$, $mn-k^2=nk-n-k^2>0$, so $\gamma=m\mu+\lambda$ is a positive framing by Lemma \ref{PositiveFraming}. To apply the Casson-Walker obstruction in Proposition \ref{CassonWalkerObstruction}, recall that $K'=K(n,1,k)$ is the simple knot with $[K]=[K']$ and the manifold obtained from $\gamma$-surgery along $K'$ is in fact a lens space
$$Y_\gamma(K')=M(0,0;1/(n-k),1/k,-1/1)=L(nk-n-k^2,k-1).$$
As $Y_\gamma(K)=L(s,1)=L(nk-n-k^2,1)$, (\ref{inequality:CassonWalker}) requires $$\Delta \lambda=\lambda(L(nk-n-k^2,1))-\lambda (L(nk-n-k^2,k-1))\geq 0.$$

By Proposition \ref{CaLe} and (\ref{spq}), we have
\begin{equation*}
%\begin{aligned}\label{eq1001}
\lambda(L(nk-n-k^2,1))=-\frac{1}{24}\cdot\left(\frac{2}{nk-n-k^2}+nk-n-k^2-3\right).\\
%=&\frac{1}{6}+ \frac{(nk-n-k^2)^2}{12} - \frac{(nk-n-k^2)}{4}.
%\end{aligned}
\end{equation*}
Next we compute $\lambda(L(nk-n-k^2,k-1))$. Write $\frac{nk-n-k^2}{k-1}$ as Hirzebruch-Jung continued fraction:
\begin{equation*}
\frac{nk-n-k^2}{k-1}=n-k-1-\frac{1}{k-1}.
\end{equation*}
This gives 
\begin{align*}
\lambda(L(nk-n-k^2,k-1))&=-\frac{1}{2}s(k-1,nk-n-k^2)\\&=-\frac{1}{24}\left(\frac{k-1}{nk-n-k^2}+\frac{n-k-1}{nk-n-k^2}+(n-k-1-3)+(k-1-3)\right)\\
&=-\frac{1}{24}\left(\frac{n-2}{nk-n-k^2}+n-8\right),
\end{align*}
where we used $(n-k-1)(k-1)=nk-n-k^2+1 \equiv 1 \pmod {nk-n-k^2}$. Thus
%\begin{equation}
%\begin{aligned} \label{eq1002}
%\lambda(L(nk-n-k^2,k-1))=-\frac{1}{24}\left(\frac{n-2}{nk-n-k^2}\right)-\frac{1}{24}(n-8).
%\end{aligned}
%\end{equation}

\begin{equation*} %\label{1004}
\Delta\lambda=\frac{1}{24}\left(\frac{n-4}{nk-n-k^2}+2n-nk+k^2-5\right).
\end{equation*}
For the specific cases listed in Lemma \ref{CaseFnkvalue}: 
\begin{itemize}
\item When $k=3$ and $n=6$: $\Delta\lambda=-1/18$ , %which leads to a contradiction;
\item When $k=3$ and $n=7$: $\Delta\lambda=-1/10, $%we have $W=-\frac{1}{2}$; %which leads to a contradiction;
\item When $k=3$ and $n=8$: $\Delta\lambda=-1/7$, %W=-1$; %which leads to a contradiction;
\item When $k=3$ and $n=9$: $\Delta \lambda=-5/27$, %$W=-\frac{5}{3}$; %which leads to a contradiction;
\item When $k=4$ and $n=8$: $\Delta\lambda=-3/16$. %$W=-\frac{3}{2}$. %which leads to a contradiction.
\end{itemize}
All these values contradict the non-negativity of $\Delta\lambda$, so there is no surgery between such pairs.

\end{proof}

The rest of the section studies the case of negative $s$, and we have:

\begin{prop}
\label{m=k-1s<0}
For $n \geq -s >0$ of the same parity, if the lens space $L(s,1)$ can be obtained by a distance one surgery along a knot in $L(n,1)$ with winding number $k>1$ and surgery coefficient $m=k-1$, then $n$ and $s$ must be one of the following pairs:
\begin{itemize}
\item[(1)] $n=5$ and $s=-1$;
\item[(2)] $n=6$ and $s=-2$;
\item[(3)] $n=9$ and $s=-5$;
\item[(4)] $n=9$ and $s=-9$;
\item[(5)] $n=14$ and $s=-10$.
\end{itemize}
\end{prop}

\begin{proof}

By Lemma \ref{PositiveFraming}, we have $\gamma=m\mu+\lambda$ as a positive framing since $mn-k^2=nk-n-k^2>0$.
%By (\ref{eqmn-k2}), we have 
Also, $s=-(nk-n-k^2)$ since we assumed $s<0$. %Thus the desired surgery result is $-L(nk-n-k^2,1)$.
Unlike the earlier case for positive $s$, the Casson-Walker invariants do not provide a strong obstruction here. Instead, we will apply the $d$-invariant formulas from Proposition \ref{prop rational d-inv1}.

We first identify the self-conjugate $\rm Spin^c$ structures on 
$Y_\gamma(K)=-L(nk-n-k^2,1)$ and $M=Y_\gamma(K')=L(nk-n-k^2,k-1).$ 
Depending on parities, we conduct a case-by-case analysis below.

\medskip
\noindent
\textit{Case 1}: \textit{$n$ and $s$ are} \textit{odd}.  
Denote the unique self-conjugate $\rm Spin^c$ structure on $Y_\gamma(K)= -L(nk-n-k^2,1)$ by 
$$\mathfrak{t}=0 \in {\rm Spin^c} (-L(nk-n-k^2,1)).$$

\noindent
\textit{Subcase 1a}: \textit{$k$ is even}.  By Lemma \ref{CaseFnkvalue}, $k=2$ and $n\geq 5$ is an odd integer.  Then $L(nk-n-k^2, k-1)=L(n-4,1)$, 
and the unique self-conjugate $\rm Spin^c$ structure is 
$$\mathfrak{t}^M=0\in {\rm Spin^c} (L(n-4,1)).$$ 
We compute: $$d(Y_\gamma(K), \mathfrak{t})=d(-L(n-4,1),0)= \frac{5-n}{4}. $$ and 
$$d(M, \mathfrak{t}^M)= d(L(n-4, 1), 0)=\frac{n-5}{4}.$$
Consequently
\begin{equation*}
N_0=\frac{n-5}{4}
\end{equation*}
is the non-negative integer in Proposition \ref{prop rational d-inv1}.

For $N_0=0$, we have $n=5$ and $s=4-n=-1$.  This confirms the existence of a surgery from $L(5,1)$ to $L(-1,1)=S^3$, giving \textbf{Item (1)} in the statement.  

For $N_0=1$, we have $n=9$ and $s=-5$. This gives \textbf{Item (3)} in the statement.  

%When $l=0$ and $n=5$, we have $-L(n-4,1)=-L(1,1)=S^3$. The distance one surgery from $L(5,1)$ to $S^3$ does exist, which can be realized as the double branched cover of the band surgery shown as Figure \ref{bandsurgery5}.

%When $l=1$ and $n=9$, we have $-L(n-4,1)=-L(5,1)$. Our method cannot obstruct surgery from $L(9,1)$ to $-L(5,1)$. 

For $N_0 \geq 2$, we apply Formula (\ref{dinv2}). By tracing $\mu$ using Figure \ref{Fig:mu}, we find that $\mathfrak{t}^M-i^{\ast}PD[\mu]$ corresponds to the $\rm Spin^c$ structure $2$ on $L(n-4,1)$. Formula (\ref{dinv2}) then gives
\[-d(L(n-4,1),j)=d(L(n-4,1),2)-2N_1,\]
for some integer $j$. By (\ref{eq2}), we obtain 
\begin{equation}
\label{VEQ101}
N_1=\frac{(8-n)^2+(2j-n+4)^2}{8(n-4)}-\frac{1}{4},
\end{equation}
where $N_1$ equals either $N_0=\frac{n-5}{4}$ or $N_0-1=\frac{n-9}{4}$. 

For $N_1=N_0=\frac{n-5}{4}$,
Equation (\ref{VEQ101}) is reduced to
\[j^2-(n-4)j+12-2n=0.\]

For $N_1=N_0-1=\frac{n-9}{4}$, Equation (\ref{VEQ101}) is reduced to
\[j^2-(n-4)j+4=0.\]
Simple algebra shows neither equation admits integer solutions; hence, no valid surgery exists in this case.

\begin{figure}[!h]
\centering
\includegraphics[width=5.5in]{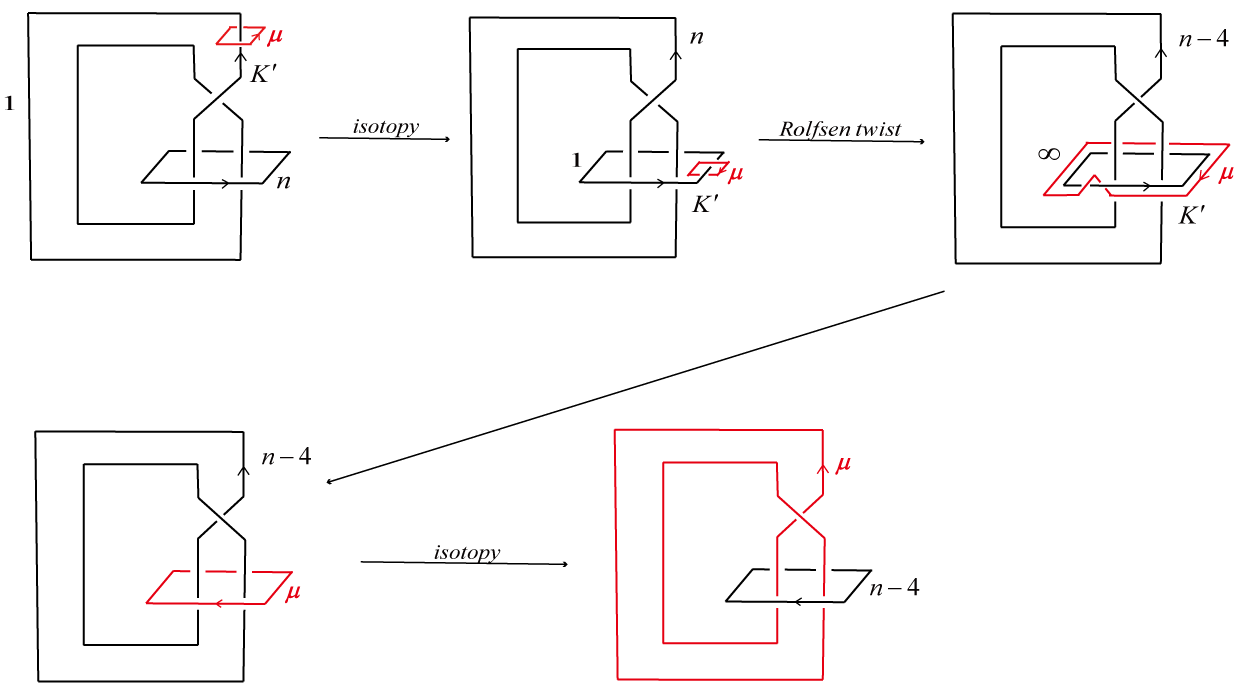}
\caption{Tracing the meridian $\mu$ of the simple knot $K'=K(n,1,2)$ in $M=L(n-4, 1)$.}
\label{Fig:mu}
\end{figure}

\medskip
\noindent
\textit{Subcase 1b: $k$ is odd.}   By Lemma \ref{CaseFnkvalue}, $k=3$, and $n=7$ or $9$.  The lens space $L(nk-n-k^2, k-1)=L(2n-9,2)$
and  $$\mathfrak{t}^M=n-4\in{\rm Spin^c} (L(2n-9,2))$$ is the unique self-conjugate
$\rm Spin^c$ structure. We compute $$d(Y_\gamma(K), \mathfrak{t})=d(-L(2n-9,1),0)= \frac{5-n}{2}; $$ and for odd $n$,
$$d(M, \mathfrak{t}^M)= d(L(2n-9, 2), n-4)=0.$$
Consequently
\begin{equation*}
N_0=\frac{n-5}{4}.
\end{equation*}

When $n=7$, $N_0=1/2$ is not an integer and contradicts to Proposition \ref{prop rational d-inv1}.

When $n=9$, we have $N_0=1$ and $s=-9$.  This is the \textbf{Item (4)} in the statement.  

\medskip
\noindent
\textit{Case 2: $n$ and $s$ are even.}
In this case, $k$ is either $2$ or $4$ by Lemma \ref{lemmaparity} and \ref{CaseFnkvalue}.
Denote the two self-conjugate $\rm Spin^c$ structures on $Y_\gamma(K)=-L(nk-n-k^2,1)$ by
$$\mathfrak{t}_1=0, \mathfrak{t}_2=\frac{nk-n-k^2}{2} \in {\rm Spin^c} (-L(nk-n-k^2,1)).$$
There are also two self-conjugate $\rm Spin^c$ structures on $Y_\gamma(K')=L(nk-n-k^2,k-1)$
$$\frac{k-2}{2},\; \frac{nk-n-k^2+k-2}{2} \in {\rm Spin^c} (L(nk-n-k^2,k-1)).$$
We have two possibilities:

\noindent
(a) $\mathfrak{t}^M_1=\frac{k-2}{2}$,  $\mathfrak{t}^M_2=\frac{nk-n-k^2+k-2}{2}$; or \\
(b)  $\mathfrak{t}^M_1=\frac{nk-n-k^2+k-2}{2}, \mathfrak{t}^M_2=\frac{k-2}{2}$.

\medskip
\noindent
\textit{Subcase 2a: $\mathfrak{t}^M_1=\frac{k-2}{2}$ and  $\mathfrak{t}^M_2=\frac{nk-n-k^2+k-2}{2}$.}

\begin{itemize}

\item When $k=2$, we apply Proposition \ref{prop rational d-inv1} with 
\[\mathfrak{t}_1=0 \in {\rm Spin^c} (-L(n-4,1)) \,\,\, {\rm and} \,\,\, \mathfrak{t}_1^M=0\in {\rm Spin^c} (L(n-4,1)),\]
and the same calculation as in Subcase 1a gives $N_0=\frac{n-5}{4}$. When $n$ is even, $N_0$ is not an integer, leading to a contradiction.  

\item When $k=4$, we have $n=8$ and $s=-8$, with the desired surgery result being $L(-8,1)$ and a linking form of $-1/8$. The linking form of $Y_{\gamma}(K)$ must match that of $M=L(nk-n-k^2,k-1)=L(8,3)$, which is $3/8$, since linking forms depend only on homological information. However, the two linking forms $-1/8$ and $3/8$ are not isomorphic, as $-3$ is not a quadratic residue modulo $8$. Hence, this case is impossible.

\end{itemize}
 
\medskip
\noindent
\textit{Subcase 2b: $\mathfrak{t}^M_1=\frac{nk-n-k^2+k-2}{2}$ and $\mathfrak{t}^M_2=\frac{k-2}{2}$.}  

\begin{itemize}
    
\item When $k=2$, we apply Proposition \ref{prop rational d-inv1} with 
\[\mathfrak{t}_1=0 \in {\rm Spin^c} (-L(n-4,1)) \,\,\, {\rm and} \,\,\, \mathfrak{t}_1^M=\frac{n-4}{2}\in {\rm Spin^c} (L(n-4,1)). \]
We find $N_0=\frac{n-6}{8}$, which must be a non-negative integer.

For $N_0=0$, we have $n=6$ and $s=-2$. This gives \textbf{Item (2)} in the statement. 

For $N_0=1$, we have $n=14$ and $s=-10$. This gives \textbf{Item (5)} in the statement. 

For $N_0\geq 2$, we apply Formula (\ref{dinv2}). From Figure \ref{Fig:mu}, $\mathfrak{t}_1^M+i^{\ast}PD[\mu]$ corresponds to the $\rm Spin^c$ structure $\frac{n-8}{2}$ on $L(n-4,1)$. 
Thus, Formula (\ref{dinv2}) gives
\begin{equation*}
%\label{22eq}
-d(L(n-4,1),j)=d(L(n-4,1),\frac{n-8}{2} )-2N_1,
\end{equation*}
for some integer $j$. From equation (\ref{eq2}),
we derive 
\begin{equation}
\label{2003eq}
N_1=\frac{(2j-(n-4))^2+16}{8(n-4)}-\frac{1}{4},
\end{equation}
where $N_1$ equals either $N_0=\frac{n-6}{8}$ or $N_0-1=\frac{n-14}{8}$. 

For $N_1=N_0=\frac{n-6}{8}$, Equation (\ref{2003eq}) is reduced to
$$j^2-(n-4)j+4=0.$$

For $N_1=N_0-1=\frac{n-14}{8}$, Equation (\ref{2003eq}) is reduced to
$$j^2-(n-4)j+2n-4=0.$$
Simple algebra shows neither equation admits integer solutions; hence, no valid surgery exists in this case.  

\item When $k=4$, this case can be obstructed using linking forms, as previously discussed.

%For Case C, the desired surgery result is $L(-8,1)$, while $(m=3)$-surgery along the simple knot $K(8,1,4) \subset L(8,1)$ yields $L(8,3)$. As discussed in Cases 1.(1) and 2.(1): Case C, {\bf the linking forms of the two lens spaces are not equivalent}, which provides an obstruction for this case.

\end{itemize}

\end{proof}

\noindent
In summary, we have proven:

\begin{prop}
\label{m=k-1n}
If the lens space $L(s,1)$ can be obtained by a distance one surgery along a knot with winding number $k>1$ in $L(n,1)$ and surgery coefficient $m=k-1$, where $n \geq |s|>0$ have the same parity, then $n$ and $s$ must satisfy one of the following cases:
\begin{itemize}
\item[(1)] $n \geq 5$ and $s=n-4$;
\item[(2)] $n=5$ and $s=-1$;
\item[(3)] $n=6$ and $s=-2$;
\item[(4)] $n=9$ and $s=-5$;
\item[(5)] $n=9$ and $s=-9$;
\item[(6)] $n=14$ and $s=-10$.
\end{itemize}
\end{prop}

\subsection{Case (F) $k>1$ and $m \leq k-3$}
\label{mleqk-3s}

By Theorem \ref{Thm:L-spaceknotpartialorder}, the Seifert fibered space 
$$M=M\left(0,0;\frac{1}{n-k},\frac{1}{k},\frac{1}{m-k}\right)$$ 
obtained by $\gamma$-surgery on the simple knot $K'=K(n,1,k)$ must be an $L$-space. We will apply the classification of Seifert fibered $L$-spaces based on Lisca and Stipsicz \cite[Theorem 1.1]{LS}. See also \cite[Section 3.1]{ChaKof} for an alternative version of the statement in terms of quasi-alternating pretzel links.

\begin{lemma}
\label{SFL}
The Seifert fibered space $M(0,0;1/p_1,1/p_2,-1/p_3)$ with $p_1$, $p_2$, $p_3 \geq 2$ is an $L$-space if and only if either
\begin{enumerate}
\item $p_3 \geq \min\{p_1,p_2\}$, or 
\item $p_3=\min\{p_1,p_2\}-1$ and $\max\{p_1,p_2\}\leq2p_3+1$.
\end{enumerate}
\end{lemma}

In our case, we assumed $k-m \geq 3$ and $n-k \geq k \geq 2$. Therefore, $M$ is an $L$-space if and only if 

\begin{enumerate}
\item $m \leq 0$  or
\item $m=1$ and $k \geq \frac{n+1}{3}$.
\end{enumerate}
Furthermore, we assumed $|s|\leq n$. However, for $m \leq -1$, we find $$|s|=|mn-k^2| \geq n+k^2>n.$$ This contradiction reduces our analysis to two possible cases:

\begin{enumerate}
\item $m=0$ 
\item $m=1$ and $k \geq \frac{n+1}{3}$.
\end{enumerate}

\begin{prop}
\label{mleqk-3n}

For $n \geq |s|>0$ of the same parity, the lens space $L(s,1)$ cannot be obtained by a distance one surgery along a knot of winding number $k>1$ in $L(n,1)$ with surgery coefficient $m\leq k-3$.
%The lens space $L(s,1)$ cannot be obtained by a distance one surgery along a knot with winding number $k>1$ in $L(n,1)$ and surgery coefficient $m \leq k-3$, where $n>0$ and $s \neq 0$ have the same parity and $|s| \leq n$. 
\end{prop}

\begin{proof}

With the constraint on the values of $m, n$ and $k$ as stated above, we can prove the statement based on the Casson-Walker obstruction from Proposition \ref{CassonWalkerObstruction}.

\medskip
\noindent
\textit{Case 1}: $m=0$.  %Since $m \leq k-3$ and $k\leq \frac{n}{2}$, we have $k \geq 3$ and $n \geq 6$. 
Note that $mn-k^2=-k^2<0$, so $\gamma=m\mu+\lambda$ is a negative framing. 
We reverse the orientation and consider $L(-s,1)$ as being obtained from a distance one surgery along a knot in $L(-n,1)$, which corresponds to a surgery with a positive framing.

Note that we can have either $s > 0$ or $s < 0$. For positive $s$, we have:

\medskip
\noindent
\textit{Subcase 1a}: $s=k^2>0$. 
The Casson-Walker obstruction (\ref{inequality:CassonWalker}) requires 
\begin{equation}\label{1a}
\Delta\lambda = \lambda(L(-s,1))-\lambda(-M)   
              =\lambda(M)-\lambda(L(k^2,1)) 
              \geq 0.
\end{equation}

On the other hand, we have
\begin{equation} \label{lambdaL0}
\lambda(L(k^2,1))=-\frac{1}{24}\cdot\left(\frac{2}{k^2}+k^2-3\right).
\end{equation}
Substituting $m=0$ into (\ref{lambdaM}), we obtain
\begin{equation}\label{lambdam0}
\lambda(M)=-\frac{n-k}{24}+\frac{n-k}{12k^2}-\frac{n-k}{24}
=\frac{n-k}{12k^2}-\frac{n-k}{12}.
\end{equation}

This gives $$\Delta\lambda=\frac{(k^2-1)(k^2-2+2k-2n)}{24k^2}. $$
As we assumed $k^2=s\leq n$ and $1<k\leq \frac{n}{2}$, we find $\Delta \lambda<0$, which contradicts (\ref{1a}).

\medskip
\noindent
\textit{Subcase 1b}: $s=-k^2<0$. The obstruction (\ref{inequality:CassonWalker}) requires 
\begin{equation}\label{1b}
\Delta\lambda = \lambda(L(-s,1))-\lambda(-M)   
              =\lambda(M)+\lambda(L(k^2,1)) 
              \geq 0.
\end{equation}
From (\ref{lambdaL0}) and (\ref{lambdam0}), we obtain
\begin{equation*}
\Delta\lambda=\frac{(k^2-1)(2-k^2+2k-2n)}{24k^2}<0,
\end{equation*}
which also leads to a contradiction.

\medskip
\noindent
\textit{Case 2: $m=1$ and $k \geq \frac{n+1}{3}$}. Recall the assumption $m\leq k-3$ in Case \textbf{(G)}, so $k\geq 4$ and consequently $n\geq 2k =8$. Given $k \geq \frac{n+1}{3}$,
We have $|s|=|mn-k^2|=k^2-n>0$. The condition $|s|\leq n$  implies
\[\left(\frac{n+1}{3}\right)^2 \leq k^2 \leq 2n.\]

Simple calculation restricts $(n,k)$ to the following possibilities
\begin{itemize}
\item[(i)] $k=4$ and $8 \leq n \leq 11$;
\item[(ii)] $k=5$ and $13 \leq n \leq 14$. 
\end{itemize}

Note that $\gamma=m\mu+\lambda$ is a negative framing for $m=1$ since $nm-k^2=n-k^2<0$. We reverse the orientation and consider $-L(s,1)$ as being obtained from a positive surgery along a knot in $L(-n,1)$. 

We can have either $s>0$ or $s<0$. For positive $s$, we have:

\medskip
\noindent
\textit{Subcase 2a}: $s=k^2-n>0$. The Casson-Walker obstruction (\ref{inequality:CassonWalker}) requires 
\begin{equation}\label{2a}
\Delta\lambda = \lambda(L(-s,1))-\lambda(-M)   
              =\lambda(M)-\lambda(L(k^2-n,1)) 
              \geq 0.
\end{equation}
On the other hand, we know
\begin{equation} \label{lambdaL}
\lambda(L(k^2-n,1))
=-\frac{1}{24}\left(\frac{2}{k^2-n}+k^2-n-3 \right).
\end{equation}
Substituting $m=1$ into (\ref{lambdaM}), we obtain
\begin{equation}\label{lambdam}
\begin{aligned}
\lambda(M)&=-\frac{(k-1)(n-k)k}{24(k^2-n)}+\frac{1+n-k}{12(k^2-n)}-\frac{1+n-k}{24}.\\
%&=\frac{(k-1)(n-k)k}{12}-\frac{1+n-k}{6}+\frac{(1+n-k)(k^2-n)}{12}.
\end{aligned}
\end{equation}
%Note that in this case $mn-k^2<0$, which is the case we considered when computing $\lambda(M)$ in Section \ref{The Casson-Walker invariant for Seifert fibered spaces}, i.e., when the Euler number $e>0$. 
This yields
\begin{equation*}
\Delta\lambda=\frac{k^4+2k^3-5k^2-4nk^2-2k+2n^2+6n+4}{24(k^2-n)}.
\end{equation*}
For $k=4$ and $8 \leq n \leq 11$, 
\begin{equation*}
\Delta\lambda=\frac{2n^2-58n+300}{24(16-n)}<0.
\end{equation*}
For $k=5$ and $13 \leq n \leq 14$, 
\begin{equation*}
\Delta\lambda=\frac{2n^2-94n+744}{24(25-n)}<0.
\end{equation*}
Both contradict (\ref{2a}).

\medskip
\noindent
\textit{Subcase 2b}: $s=n-k^2<0$. The obstruction requires 
\begin{equation*}\label{2b}
\Delta\lambda = \lambda(L(-s,1))-\lambda(-M)   
              =\lambda(M)+\lambda(L(k^2-n,1)) 
              \geq 0.
\end{equation*}
Using (\ref{lambdaL}) and (\ref{lambdam}):
\begin{equation*}
\Delta\lambda=\frac{k(k-2)(1-k^2)}{24(k^2-n)}<0
\end{equation*}
This final contradiction completes the proof. 

\end{proof}


\begin{thebibliography}{10}
\bibitem[Ber18]{Berge} J. Berge, \textit{Some knots with surgeries yielding lens spaces.} arXiv:1802.09722, 2018.
%\bibitem{Bleiler} S. A. Bleiler and R. A. Litherland, \textit{Lens spaces and Dehn surgery.} Proc. Amer. Math. Soc., \textbf{107}(4): 1127-1131, 1989.
\bibitem[BBCW12]{Boil} M. Boileau, S. Boyer, R. Cebanu and G. S. Walsh,  \textit{Knot commensurability and the Berge conjecture.} Geom. Topol., \textbf{16}(2): 625-664, 2012.
\bibitem[BL90]{BL} S. Boyer and D. Lines, \textit{Surgery formulae for Casson's invariant and extensions to homology lens spaces.} J. Reine Angew. Math., \textbf{405}: 181-220, 1990.
\bibitem[CK09]{ChaKof} A. Champanerkar and I. Kofman, \textit{Twisting quasi-alternating links.} Proc. Amer. Math. Soc., \textbf{137}(7): 2451–2458, 2009.
\bibitem[CH15]{Cochran} T. D. Cochran and P. D. Horn, \textit{Structure in the bipolar filtration of topologically slice knots.} Algebr. Geom. Topol., \textbf{15}(1): 415–428, 2015.
\bibitem[CFH16]{CFH} A. Conway and S. Friedl, G. Herrmann, \textit{Linking forms revisited.} Pure Appl. Math. Q. \textbf{12}(4): 493-515, 2016
\bibitem[DS00]{Darcysumner} I. K. Darcy and D. W. Sumners, \textit{Rational tangle distances on knots and links.} Math. Proc. Cambridge Philos. Soc., \textbf{128}(3): 497-510, 2000.
\bibitem[Gai18]{Gai} F. Gainullin, \textit{Heegaard Floer homology and knots determined by their complements.} Alg. Geom. Topol., \textbf{18}(1): 69-109, 2018.
%\bibitem{Goda} H. Goda and M. Teragaito, \textit{Dehn surgeries on knots which yield lens spaces and genera of knots.} Math. Proc. Cambridge Philos. Soc., \textbf{129}(3): 501-515, 2000.
\bibitem[Gre13]{Greene} J. E. Greene, \textit{The lens space realization problem.} Ann. of Math., \textbf{177}(2):449-511, 2013.
%\bibitem{Greene2} J. E. Greene, \textit{A spanning tree model for the Heegaard Floer homology of a branched double‐cover.} J. Topol., \textbf{6}(2): 525-567, 2013.
\bibitem[Hat07]{Hat} A. Hatcher, \textit{Notes on basic 3-manifold topology.} https://pi.math.cornell.edu/~hatcher/3M/3M. pdf, 2007.
\bibitem[Kir97]{Kirby} R. Kirby, \textit{Problems in low-dimensional topology.} Geometric topology (Athens, GA, 1993)(Rob Kirby, ed.), AMS/IP Stud. Adv. Math. 2, Amer. Math. Soc., Providence, RI: 35–473, 1997.
%\bibitem {Lick} W. R. Lickorish, \textit{A representation of orientable combinatorial 3-manifolds.} Ann. of Math., \textbf{76}(2): 531-540, 1962.
\bibitem[Les96]{Les} C. Lescop, \textit{Global surgery formula for the Casson-Walker invariant.} Princeton University Press, 1996.
\bibitem[LMV19]{LMV} T. Lidman, A. H. Moore and M. Vazquez,  \textit{Distance one lens space fillings and band surgery on the trefoil knot.} Alg. Geom. Topol., \textbf{19}(5): 2439-2484, 2019.
\bibitem[Lin17]{Lin} F. Lin, \textit{The surgery exact triangle in Pin(2)–monopole Floer homology.} Alg. Geom. Topol., \textbf{17}(5): 2915-2960, 2017.
\bibitem[LS07]{LS} P. Lisca and A. Stipsicz, \textit{Ozsv\'{a}th-Szab\'{o} invariants and tight contact 3-manifolds, \uppercase\expandafter{\romannumeral3}.} J. Symplectic Geom., \textbf{5}(4): 357-384, 2007.
\bibitem[Liv20]{Livingston} C. Livingston, \textit{Chiral smoothings of knots.} Proc. Edinburgh Math. Soc., \textbf{63}(4): 1048-1061, 2020.
\bibitem[MV20]{MV} A. H. Moore and M. Vazquez,  \textit{A note on band surgery and the signature of a knot.} Bull. London Math. Soc., \textbf{52}(6): 1191-1208, 2020.
\bibitem[Ni09]{Ni} Y. Ni, \textit{Link Floer homology detects the Thurston norm.} Geom. Topol., \textbf{13}(5): 2991-3019, 2009.
\bibitem[NW15]{NiWu} Y. Ni and Z. Wu,  \textit{Cosmetic surgeries on knots in $S^3$.} J. Reine Angew. Math., \textbf{706}: 1-17, 2015.
\bibitem[OS03a]{OS1} P. Ozsv\'{a}th and Z. Szab\'{o}, \textit{Absolutely graded Floer homologies and intersection forms for four-manifolds with boundary.} Adv. Math., \textbf{173}(2): 179-261, 2003.
\bibitem[OS03b]{OSd} P. Ozsv\'{a}th and Z. Szab\'{o}, \textit{On the Floer homology of plumbed three-manifolds.} Geom. Topol., \textbf{7}(1): 185-224, 2003.
\bibitem[OS05]{OSkl} P. Ozsv\'{a}th and Z. Szab\'{o}, \textit{On knot Floer homology and lens space surgeries.} Topology, \textbf{44}(6): 1281-1300, 2005.
\bibitem[OS06]{OSc} P. Ozsv\'{a}th and Z. Szab\'{o}, \textit{Holomorphic triangles and invariants for smooth four-manifolds.} Adv. Math., \textbf{202}(2): 326-400, 2006.
%\bibitem{OSi} P. Ozsv\'{a}th and Z. Szab\'{o},  \textit{Knot Floer homology and integer surgeries.} Alg. Geom. Topol., \textbf{8}(1): 101-153, 2008.
\bibitem[OS10]{OSr} P. Ozsv\'{a}th and Z. Szab\'{o},  \textit{Knot Floer homology and rational surgeries.} Alg. Geom. Topol., \textbf{11}(1): 1-68, 2010.
\bibitem[Ras03]{Ras1} J. Rasmussen, \textit{Floer homology and knot complements.} Ph.D. Thesis, arXiv:math/0306378, 2003.
\bibitem[Ras04]{Ras3} J. Rasmussen, \textit{Lens space surgeries and a conjecture of Goda and Teragaito.}  Geom. Topol., \textbf{8}(3): 1013-1031, 2004.
\bibitem[Ras07]{Ras2} J. Rasmussen,  \textit{Lens space surgeries and L-space homology spheres.} arXiv:0710.2531, 2007.
\bibitem[RR17]{RasJSD} J. Rasmussen and S. D. Rasmussen, \textit{Floer simple manifolds and L-space intervals.} Adv. Math., \textbf{322}: 738-805, 2017.
\bibitem[Rus04]{Rus} R. Rustamov, \textit{Surgery formula for the renormalized Euler characteristic of Heegaard Floer homology.} arXiv:math/0409294v3, 2004.
%\bibitem{Tur} V. Turaev, \textit{Torsions of 3-dimensional manifolds.} Progress in Math. 208, $\rm Birkh\ddot{a}user$ Verlag, Basel, MR1958479, 2002.
%\bibitem {Wal} A. H. Wallace, \textit{Modifications and cobounding manifolds.} Canad. J. Math. \textbf{12}: 503-528, 1960.
\bibitem[WY21]{WY} Z. Wu and J. Yang, \textit{Studies of distance one surgeries on the lens space $L(p, 1)$.} Math. Proc. Cambridge Philos. Soc., 1-35, 2021.
\bibitem[WY23]{WY2} Z. Wu and J. Yang, \textit{Rational genus and Heegaard Floer homology} arXiv: 2307.06807
\bibitem[Ye23]{Ye} F. Ye, \textit{Constrained knots in lens spaces.} Alg. Geom. Topol., \textbf{23}(3): 1097-1166, 2023.
\bibitem[Zek15]{Zek} A. Zekovi\'{c}, \textit{Computation of Gordian distances and H2-Gordian distances of knots.} Yugosl. J. Oper. Res., \textbf{25}(1): 133-152, 2015.

\end{thebibliography}
\end{document}